\documentclass[12pt]{amsart}
\usepackage{amsfonts,amssymb,a4wide,url,mathscinet,mathdots,mathtools}
\usepackage[all]{xypic}
\newcommand{\bs}{\backslash}
\newcommand{\Z}{\mathbb{Z}}
\newcommand{\Q}{\mathbb{Q}}
\newcommand{\R}{\mathbb{R}}
\newcommand{\C}{\mathbb{C}}
\newcommand{\A}{\mathbb{A}}
\newcommand{\inorm}{\operatorname{N}}
\newcommand{\nnn}{\mathfrak{n}}
\newcommand{\fin}{\operatorname{fin}}
\newcommand{\level}{\operatorname{level}}
\newcommand{\can}{\operatorname{Can}}

\renewcommand{\Re}{\operatorname{Re}}
\renewcommand{\Im}{\operatorname{Im}}
\newcommand{\AF}{{\mathcal A}}
\newcommand{\sprod}[2]{\left\langle#1,#2\right\rangle}
\newcommand{\rest}{\big|}
\newcommand{\Eisen}{\mathcal{E}}
\newcommand{\CC}{(CC)}
\newcommand{\aaa}{\mathfrak{a}}

\newcommand{\Tr}{\operatorname{Tr}}
\newcommand{\OO}{\mathcal{O}}
\newcommand{\vol}{\operatorname{vol}}
\newcommand{\Gal}{\operatorname{Gal}}
\newcommand{\GF}{($\Gamma$F)}
\newcommand{\Disc}{\Delta}
\newcommand{\FE}{(FE)}
\newcommand{\LG}{{^L}G}
\newcommand{\LP}{{^L}P}
\newcommand{\LM}{{^L}M}
\newcommand{\LU}{{^L}U}
\newcommand{\expcond}{\mathfrak{e}}
\newcommand{\Sym}{\operatorname{Sym}}
\newcommand{\Asai}{\operatorname{As}}
\newcommand{\swrz}{\mathcal{S}}
\newcommand{\JL}{\operatorname{JL}}
\newcommand{\Ar}{\operatorname{Ar}}
\newcommand{\Std}{\operatorname{St}}
\newcommand{\aut}{\operatorname{aut}}
\newcommand{\disc}{\operatorname{disc}}
\newcommand{\cusp}{\operatorname{cusp}}
\newcommand{\GL}{\operatorname{GL}}
\newcommand{\SL}{\operatorname{SL}}
\newcommand{\Res}{\operatorname{Res}}
\newcommand{\red}{\operatorname{red}}
\newcommand{\abs}[1]{\left|{#1}\right|}
\newcommand{\norm}[1]{\lVert#1\rVert}
\newcommand{\iii}{\mathrm{i}}
\newcommand{\modulus}{\delta}
\newcommand{\Ad}{\operatorname{Ad}}
\newcommand{\arithcond}{\mathfrak{n}}
\newcommand{\infcond}{\mathfrak{c}}
\newcommand{\cond}{\mathfrak{c}}
\newcommand{\Lie}{\operatorname{Lie}}
\newcommand{\unilen}{l} % length of unipotent
\newcommand{\prpr}[1]{\hat #1}
\newcommand{\modM}{\tilde M_{\alpha}}

\newcommand{\K}{\mathbf{K}}
\newcommand{\id}{\operatorname{id}}
\newcommand{\PPP}{\mathcal{P}}
\newcommand{\rts}{\Sigma}
\newcommand{\SC}{\operatorname{sc}}
\newcommand{\scprj}{p^{\SC}}
\newcommand{\diag}{\operatorname{diag}}
\newcommand{\der}{\operatorname{der}}
\newcommand{\param}{\Lambda}
\newcommand{\Crtn}{\mathfrak{h}}
\newcommand{\Lieg}{\mathfrak{g}}
\newcommand{\Tate}{\operatorname{Tate}}
\newcommand{\GJ}{\operatorname{GJ}}
\newcommand{\JPSS}{\operatorname{JPSS,Sh}}
\newcommand{\PSR}{\operatorname{PSR}}
\newcommand{\Sh}{\operatorname{Sh}}
\newcommand{\gnr}{\bullet}
\newcommand{\sm}[4]{\left(\begin{smallmatrix}{#1}&{#2}\\{#3}&{#4}\end{smallmatrix}\right)}
\newcommand{\Nrd}{\operatorname{Nrd}}
\newcommand{\Trd}{\operatorname{Trd}}
\newcommand{\F}{\mathbb{F}}

\newtheorem{theorem}{Theorem}[section]
\newtheorem{lemma}[theorem]{Lemma}
\newtheorem{definition}[theorem]{Definition}
\newtheorem{proposition}[theorem]{Proposition}%[subsection]
\newtheorem{corollary}[theorem]{Corollary}%[subsection]
\theoremstyle{remark}
\newtheorem{remark}[theorem]{Remark}%[subsection]

\begin{document}

\title[Analytic properties of intertwining operators I]{On the analytic properties of intertwining operators I: global normalizing factors}
\date{\today}
\dedicatory{To Freydoon Shahidi, for his upcoming 70th birthday}
\author{Tobias Finis}
\address{Universit\"{a}t Leipzig, Mathematisches Institut, Postfach 100920, D-04009 Leipzig, Germany}
\email{finis@math.uni-leipzig.de}
\author{Erez Lapid}
\address{Department of Mathematics, Weizmann Institute of Science, Rehovot 7610001, Israel}
\thanks{Second named author partially supported by a grant from the Minerva Stiftung}
\email{erez.m.lapid@gmail.com}

\setcounter{tocdepth}{1}

\begin{abstract}
We provide a uniform estimate for the $L^1$-norm (over any interval of bounded length) of the logarithmic derivatives
of global normalizing factors associated to intertwining operators for the following reductive groups over number fields:
inner forms of $\GL(n)$; quasi-split classical groups and their similitude groups; the exceptional group $G_2$.
This estimate is a key ingredient in the analysis of the spectral side of Arthur's trace formula.
In particular, it is applicable to the limit multiplicity problem studied by the authors in earlier papers.
\end{abstract}

\maketitle
\tableofcontents

\section{Introduction}
In this paper we study the analytic properties of the global intertwining operators associated to parabolic subgroups
of reductive groups $G$ over number fields $F$.
In the previous papers \cite{MR3352530} (joint with Werner M\"{u}ller) and \cite{1504.04795},
we defined certain properties (TWN) and (BD) pertaining to these intertwining operators,
and showed that these two properties together imply the solution of the limit multiplicity problem for congruence subgroups
of lattices contained in $G(F)$.
Property (TWN) is a global property concerning the scalar-valued normalizing factors, while (BD) is essentially a local property.
In \cite{MR3352530}, these properties were verified for the groups $\GL (n)$ and $\SL(n)$.
%%%and it is easy to extend the proof of property (BD) to inner forms of these groups (see below). \Erez{Do we do it?}

This paper is devoted to establishing property (TWN) in a number of other cases, namely for inner forms of the groups $\GL(n)$ and $\SL(n)$,
for quasi-split classical groups and their similitude groups, and for the exceptional group $G_2$.
In fact, we prove without any more effort a finer estimate (already mentioned in \cite{MR3352530}), which we call property (TWN+).
In addition to the application to the limit multiplicity problem, this property may be useful in the study of other asymptotic questions,
where it is important to control the dependence on the archimedean parameters.
The main examples of such problems are Weyl's law with remainder term and the asymptotics of Hecke operators in general families
\cite{MR2541128, 1310.6525, 1505.07285}.
In any case, the results of this paper are sufficient to extend the limit multiplicity results of
\cite{MR3352530,1504.04795} to inner forms of $\GL(n)$ and $\SL(n)$.
This will be discussed in a future paper, where we will also study property (BD) for general groups $G$, and prove a weaker variant.
This, together with the results of the current paper, will allow for an extension of the limit multiplicity results to quasi-split classical groups
and their similitude groups and $G_2$ (at least under a technical restriction on the congruence subgroups in question).

We sketch the definition of property (TWN+), which is explained in more detail in \S \ref{TWN} below.
As usual, fix a minimal Levi subgroup $M_0$ of $G$ defined over $F$, and consider a proper Levi subgroup $M$ of $G$ containing $M_0$ and a root
$\alpha \in \rts_M$.
Let $U_\alpha$ be the unipotent subgroup of $G$ corresponding to $\alpha$, $M_\alpha$ the group generated by $M$ and $U_{\pm\alpha}$,
and $\prpr{M_\alpha}$ its $F$-simple normal subgroup generated by $U_{\pm\alpha}$.
The group $\modM=\prpr{M_\alpha}\cap M$ is a maximal Levi subgroup of $\prpr{M_\alpha}$.
%Let $\widetilde{\scprj}:\prpr{M_\alpha}^{\SC}\rightarrow\prpr{M_\alpha}$ be the universal covering of $\prpr{M_\alpha}$
%and let $\scprj$ be the restriction of $\widetilde{\scprj}$ to the inverse image of $\modM$.
For $\pi\in\Pi_{\disc}(M(\A))$ let $n_\alpha(\pi,s)$ be the normalizing factor for the
global intertwining operators associated to pairs of parabolic subgroups of $G$ adjacent along $\alpha$.
These factors are meromorphic functions of finite order of the complex variable $s$ and satisfy the functional equation
$\abs{n_\alpha (\pi,\iii t)} = 1$ for all $t \in \R$.

We say that $G$ satisfies property (TWN+) (tempered winding numbers, strong version) if
we have the estimate
\[
\int_T^{T+1}\abs{n_\alpha'(\pi,\iii t)}\ dt\ll\log(\abs{T}+\param(\pi_\infty;\scprj)+\level(\pi;\scprj))
\]
for all $\pi\in\Pi_{\disc}(M(\A))$ and all real numbers $T$ where the implied constant depends only on $G$. Here,
$\level(\pi;\scprj)$ is a certain variant of the usual notion of the level of $\pi$ restricted to $\modM$,
and $\param(\pi_\infty;\scprj)$ measures the size of the infinitesimal character of the restriction of $\pi_\infty$ to $\modM$.
(See \S\ref{sec: relparams} for the precise definitions.)
We note that property (TWN+) implies property (TWN) introduced in \cite[Definition 5.2]{MR3352530}, and that property (TWN+)
has been shown (as a consequence of the theory of Rankin-Selberg $L$-functions) for the groups $\GL (n)$ and $\SL (n)$ in [ibid., Proposition 5.5].

Our method of proving (TWN+) for the groups listed above, is to use functoriality to transfer the problem to a well-understood problem
for $\GL (n)$. We start by an axiomatic treatment of automorphic $L$-functions in \S \ref{SectionLogDerivatives}, which we then apply
to the global normalizing factors in \S \ref{TWN}. In \S \ref{sec: goodLexamples} we will show (TWN+) for inner forms of $\GL (n)$ and $\SL (n)$
using the Jacquet-Langlands correspondence (obtained in general in \cite{MR2390289}, \cite{MR2684298}), which allows us to reduce the problem
to known properties of Rankin-Selberg $L$-functions.

In \S \ref{SectionClassical}, we will consider quasi-split classical groups.
We first consider the twisted exterior and symmetric square $L$-functions for $\GL (n)$,
as well as the Asai $L$-function for $\Res_{E/F}\GL (n)$, using known results obtained by the Langlands-Shahidi method and by the study of integral representations.
Using Arthur's work on functoriality from the classical groups to  $\GL (n)$ \cite{MR3135650},
extended by Mok to unitary groups \cite{MR3338302},
we will be able once again to reduce the remaining $L$-functions to Rankin-Selberg $L$-functions for $\GL (n)$.

Our approach is based on Arthur's work, which requires the full force of the stable twisted trace formula.
Among the prerequisites of Arthur's results are
\cite{MR1687096, MR2979862, MR1444087, MR2437683, MR3026269, MR2448443, MR2418405, 1205.1100, 1401.4569, 1401.7127,
1402.2753, 1403.1454, 1404.2402, 1406.2257, 1409.0960, 1410.1124, 1412.2565, 1412.2981}, to mention a few.
One may contemplate whether there is a different approach to the problem which avoids functoriality (and is perhaps applicable to other groups).
Unfortunately, at the moment we cannot say anything in this direction.

For the exceptional group $G_2$, which we treat in \S \ref{SectionExceptional}, we need to consider the (twisted) symmetric cube $L$-function for $\GL (2)$
which was studied by Kim--Shahidi.
%Although the exceptional groups $F_4$ and $E_6$ are probably within reach, new ideas seem to be necessary in order to treat the groups
%$E_7$ and $E_8$.

\subsection{}
It is a pleasure to dedicate this paper to Freydoon Shahidi for his upcoming 70th birthday.
Shahidi's influence on the field on automorphic forms cannot be overestimated. Needless to say, the current paper also owes a lot to his work.
%We hope that it will form a suitable contribution to this volume.
On a personal level it has always been a pleasure to interact with Freydoon and we wish him the very best.

\section{Estimates for logarithmic derivatives of $L$-functions} \label{SectionLogDerivatives}

We begin with an axiomatic treatment of automorphic $L$-functions which isolates the precise properties needed for the main estimate (see Proposition \ref{lem: logderbnd} below).

\subsection{}
Let us first recall some generalities about $L$-functions.
Let $G$ be a reductive group over a number field $F$ and let $\A$ be the ring of adeles of $F$.
Let $\A_{\fin}$ be the ring of finite adeles of $F$ and $F_\infty = F \otimes \R$.
Let $\abs{\cdot}_{\A^\times}$ be the idele norm on $\A^\times$.
Let $T_G$ be the $\Q$-split part of the (Zariski) connected component of the center of $\operatorname{Res}_{F/\Q} G$
(restriction of scalars) and let $A_G=T_G(\R)^\circ$ (topological connected component), viewed as a subgroup of $T_G(\A_\Q)$ (and hence of $G(\A)$).
Let $G(\A)^1 \subset G(\A)$ be the intersection of the kernels $\operatorname{ker} \abs{\chi}_{\A^\times}$ as $\chi$ ranges over
the $F$-rational characters of $G$. Then $G(\A)$ is the direct product of $G(\A)^1$ and $A_G$.
Let $\mathfrak{a}_G$ be the Lie algebra of $A_G$, a real vector space, and $\mathfrak{a}_G^*$ its dual space.
We set $\mathfrak{a}_{G,\C}^* = \mathfrak{a}_G^* \otimes \C$.

We write $\Pi_{\disc}(G(\A))$ for the set of equivalence classes of automorphic representations of $G(\A)$
which occur in the discrete spectrum of $L^2(A_GG(F)\bs G(\A))$.
We will also write $\Pi_{\cusp}(G(\A))$ for the subset of cuspidal representations.

For any $\pi=\otimes_v\pi_v\in\Pi_{\disc}(G(\A))$ let $S(\pi)$ be the finite set of places of $F$ such that at least one of the following conditions holds:
\begin{enumerate}
\item $v$ is archimedean.
\item $F/\Q$ is ramified at $v$.
\item $G$ is ramified at $v$, i.e., either $G$ is not quasi-split over $F_v$ or $G$ does not split over an unramified extension of $F_v$.
\item For every hyperspecial maximal compact subgroup $K_v$ of $G(F_v)$, $\pi_v$ does not have a nonzero vector invariant under $K_v$.
\end{enumerate}
(The exclusion of the finite places which are ramified for $F$ is inessential and is only made for convenience.)
We write $S(\pi)=S_\infty\cup S_f(\pi)$, where $S_\infty$ denotes the set of archimedean places of $F$ and $S_f(\pi)$
denotes the non-archimedean places in $S(\pi)$.
Let $S_{\Q, f}(\pi)$ (or simply $S_{\Q,f}$, if $\pi$ is clear from the context) be the set of rational primes which lie below the primes in $S_f(\pi)$.
Also set $S_{\Q}(\pi) = S_{\Q, f} (\pi)\cup\{\infty\}$.

Let $W_F$ be the Weil group of $F$ and let $\LG$ be the $L$-group of $G$ (cf.~\cite{MR546608}).
Let $r:\LG\rightarrow\GL(N,\C)$ be a continuous and $W_F$-semisimple $N$-dimensional representation of $\LG$.
For any $\pi\in\Pi_{\disc}(G(\A))$ and any place $v$ of $F$ outside of $S(\pi)$ we have the Hecke-Frobenius parameter $t_{\pi_v}\in\LG$.
To each such $v$ we associate the polynomial $P_v (X)=\det(1- X r(t_{\pi_v}))$ of degree $N$ and the local $L$-factor $L_v (s,\pi,r) = P_v (q_v^{-s})^{-1}$.
Since $\pi$ is unitary, the absolute values of the eigenvalues of $r (t_{\pi_v})$ are bounded by $q_v^\alpha$,
where $\alpha$ depends only on $G$ and $r$ (\cite{MR0302614}, \cite{MR546608}).
Therefore, for $S\supset S(\pi)$ the partial $L$-function
\[
L^S(s,\pi,r) = \prod_{v \notin S} L_v (s,\pi,r)
\]
is well-defined as an absolutely convergent product and holomorphic for $\Re s>\alpha+1$.
Because of the unitarity of $\pi$, we also have
\[
L^S(s,\pi,r^\vee)=\overline{L^S(\bar s,\pi,r)},
\]
where $r^\vee$ is the contragredient of $r$.
A cornerstone of the Langlands program is the (largely conjectural) meromorphic continuation of these $L$-functions with finitely many poles in $\C$.

As usual, we write $\Gamma_{\R}(s)=\pi^{-s/2}\mathbf{\Gamma}(s/2)$
and $\Gamma_{\C}(s)=2(2\pi)^{-s}\mathbf{\Gamma}(s)=\Gamma_{\R}(s)\Gamma_{\R}(s+1)$, where $\mathbf{\Gamma}(s)$ is the standard Gamma function.

\begin{definition}
We say that a pair $(G,r)$ has property \FE, if for any $\pi\in\Pi_{\disc}(G(\A))$ there exist rational functions $R_p$,
$p\in S_{\Q,f}(\pi)$, an integer $m\ge1$, a constant $C_\infty\in\C^*$ and parameters
$\alpha_1,\dots,\alpha_m,\alpha_1^\vee,\dots,\alpha_m^\vee\in\C$, such that the partial $L$-function $L^S (s,\pi,r)$
admits a meromorphic continuation to $\C$ with a functional equation of the form
\begin{equation} \label{eq: FE}
L^{S (\pi)}(s,\pi,r)=\big(\prod_{p\in S_{\Q} (\pi)}\gamma_p(s,\pi,r)\big) L^{S (\pi)}(1-s,\pi,r^\vee),
\end{equation}
or equivalently,
\[
L^{S (\pi)}(s,\pi,r)=\big(\prod_{p\in S_{\Q} (\pi)}\gamma_p(s,\pi,r)\big) \overline{L^{S (\pi)}(1-\bar{s},\pi,r)},
\]
where for each $p \in S_{\Q,f}(\pi)$, $\gamma_p (s,\pi,r) = R_p (p^{-s})$, and
\begin{equation} \label{eq: gammainfty}
\gamma_\infty(s,\pi,r)=C_\infty\prod_{i=1}^m\frac{\Gamma_{\R}(1-s+\alpha_i^\vee)}{\Gamma_{\R}(s+\alpha_i)}.
\end{equation}
\end{definition}

In general, this property is wide open except for a number of important special cases, some of which will be considered below.

\subsection{} \label{sec: indepgamma_p}
We first address the uniqueness of the terms in \eqref{eq: FE}.
This is of course standard. For completeness we give the details.

%will now explain some simple consequences of these properties. First we recall a simple fact from \cite{SerreAppendix}.
%\begin{remark}
%Suppose we have an identity
%\begin{equation} \label{eq: idengamma}
%\prod_{i=1}^m\frac{\Gamma_{\R}(1-s+\alpha_i^\vee)}{\Gamma_{\R}(s+\alpha_i)}\equiv CA^s
%\prod_{i=1}^{m'}\frac{\Gamma_{\R}(1-s+\beta_i^\vee)}{\Gamma_{\R}(s+\beta_i)}
%\end{equation}
%for some constants $C\in\C^*$, $A>0$ and
%$\alpha_1,\dots,\alpha_m,\beta_1,\dots,\beta_{m'},\alpha_1^\vee,\dots,\alpha_m^\vee,\beta_1^\vee,\dots,\beta_{m'}^\vee\in\C$.
%Then $m=m'$, $A=1$, and
%\end{remark}

\begin{lemma} \label{lem: elem}
Assume that \eqref{eq: FE} is satisfied for some fixed $r$ and $\pi$.
Then
\begin{enumerate}
\item The factors $\gamma_p(s,\pi,r)$, $p\in S_{\Q}(\pi)$ (and hence the rational functions $R_p$, $p\in S_{\Q,f}(\pi)$), are uniquely determined up to non-zero constant multiples.
\item The integer $m$ is uniquely determined. (We call it the degree of $\gamma_\infty$.)
\item We have
\begin{equation} \label{eq: localgammaFE}
\gamma_p(s,\pi,r) \overline{\gamma_p(1-\bar{s},\pi,r)} = \text{const.}
\end{equation}
for all $p\in S_{\Q}(\pi)$.
\item The parameters $\alpha_1^\vee,\dots,\alpha_m^\vee$ are determined by $\alpha_1,\dots,\alpha_m$ (up to permutation).
\item \label{part: condinfty} The expression
\[
\big(\prod_{i=1}^m(1+\abs{\alpha_i-1/2})(1+\abs{\alpha_i^\vee-1/2})\big)^{\frac12}
\]
is uniquely determined. It is called the archimedean conductor and denoted by $\infcond_\infty(\pi,r)$ (cf.~\cite{MR1826269}).
\item We say that $\alpha_1,\dots,\alpha_m$ are \emph{reduced}, if $\alpha_i+\alpha_j^\vee$ is not a negative odd integer for any $1 \le i,j \le m$.
We may choose the parameters $\alpha_1,\dots,\alpha_m$ to be reduced. Moreover, in this case
\begin{enumerate}
\item The zeros (resp., poles) of $\prod_{i=1}^m\ {\Gamma_{\R}(s+\alpha_i)}^{-1}$
(resp., $\prod_{i=1}^m\ {\Gamma_{\R}(1-s+\alpha_i^\vee)}$) and $\gamma_\infty(s,\pi,r)$, taken with multiplicities,
coincide.
\item The multisets $\{\alpha_1,\dots,\alpha_m\}$ and $\{\alpha_1^\vee,\dots,\alpha_m^\vee\}$ are uniquely determined.
\item As a multiset we have
\begin{equation} \label{eq: complconjalpha}
\{\alpha_1^\vee,\dots,\alpha_m^\vee\}=\{\overline{\alpha_1},\dots,\overline{\alpha_m}\}.
\end{equation}
\end{enumerate}
\end{enumerate}
%Thus, possibly after reordering, $\alpha^\vee_i = \overline{\alpha_i}$ for all $i$. \Erez{Isn't the latter true regardless?}
\end{lemma}
%Moreover, the integer $m$ in \eqref{eq: gammainfty} is uniquely determined (we call it the degree of the factor $\gamma_\infty$), and the possibilities for the parameters $\alpha_i$ and $\alpha^\vee_i$ are easily described.

\begin{proof}
Suppose that we are given a function of the form
\[
\phi(s)= C A^s \prod_{i=1}^m\frac{\Gamma_{\R}(1-s+\alpha_i^\vee)}{\Gamma_{\R}(s+\alpha_i)}
\prod_{p\in S_{\Q, f}(\pi)} \tilde{R}_p(p^{-s}),
\]
where $C$ is a non-zero constant, $A$ is a positive real number and for all $p\in S_{\Q, f}(\pi)$,
$\tilde{R}_p$ is a rational function with $\tilde{R}_p (0) = 1$.
Up to a finite multiset, the zeros of the first product are given (with multiplicities) by $-\alpha_i-2\Z^{\ge0}$, $i = 1, \ldots, m$,
while those of $\tilde{R}_p(p^{-s})$ are given by $s = (-\log x_{p,i}+2\pi\iii\Z) / \log p$, where $x_{p,1}, \ldots, x_{p,d_p}$ are the zeros of
$\tilde{R}_p$. Considering only the zeros with $\Im s$ sufficiently large, we see that the zeros of
each rational function $\tilde{R}_p$ are determined by $\phi$. Arguing similarly with the poles, we conclude that each $\tilde{R}_p$ is determined by $\phi$.

Furthermore, if we have an equality
\begin{equation} \label{eq: prodgammaidn}
\prod_{i=1}^m\frac{\Gamma_{\R}(1-s+\alpha_i^\vee)}{\Gamma_{\R}(s+\alpha_i)}=
CA^s\prod_{i=1}^{m'}\frac{\Gamma_{\R}(1-s+\beta_i^\vee)}{\Gamma_{\R}(s+\beta_i)}
\end{equation}
for some constants $A>0$ and $C$, then an examination of the zeros and poles shows that $m'=m$, and that after possibly reindexing $\alpha_1,\dots,\alpha_m$
and $\alpha_1^\vee,\dots,\alpha_m^\vee$ we may assume that
$\alpha_i-\beta_i,\alpha_i^\vee-\beta_i^\vee\in 2\Z$ for all $i$, in which case $\Gamma_{\R}(s+\alpha_i)/\Gamma_{\R}(s+\beta_i)$ and
$\Gamma_{\R}(1-s+\alpha_i^\vee)/\Gamma_{\R}(1-s+\beta_i^\vee)$ are rational functions in $s$. Thus $A=1$.

The first two parts follow. The third and fourth parts immediately follow from the first part.

Moreover, by an easy argument (e.g., \cite{SerreAppendix}) the equality \eqref{eq: prodgammaidn} holds if and only if
$A=1$, $m'=m$ and after reindexing $\alpha_1,\dots,\alpha_m$ and $\alpha_1^\vee,\dots,\alpha_m^\vee$ if necessary, there exists $0 \le k \le m$ such that
\begin{enumerate}
\item $\alpha_i+\beta_i^\vee=\alpha_i^\vee+\beta_i=1$ and $\alpha_i-\beta_i\in 2\Z$ for all $i=1,\dots,k$,
\item $\alpha_i=\beta_i$ and $\alpha_i^\vee=\beta_i^\vee$ for $i>k$, and,
\item $C=(-1)^{(\alpha_1-\beta_1+\dots+\alpha_k-\beta_k)/2}$.
\end{enumerate}
Part \ref{part: condinfty} follows.

Suppose that $\alpha_1,\dots,\alpha_m$ are not reduced, so that there exist indices $i$ and $j$ such that $\alpha_i+\alpha_j^\vee=1-2k$ for some positive integer $k$.
We may then replace $\alpha_i$ and $\alpha_j^\vee$ by $1-\alpha_j^\vee$ and $1-\alpha_i$, respectively, and multiply $C_\infty$ by $(-1)^k$.
We may repeat this process until $\alpha_1,\dots,\alpha_m$ become reduced.
The process must terminate after finitely many steps since $\sum_i(\alpha_i+\alpha_i^\vee)$ increases by $4k$ in each step.

Once $\alpha_1,\dots,\alpha_m$ are reduced, the poles of $\prod_{i=1}^m\Gamma_{\R}(1-s+\alpha_i^\vee)$ are disjoint from those of
$\prod_{i=1}^m\Gamma_{\R}(s+\alpha_i)$. Therefore, the zeros of $\prod_{i=1}^m\ {\Gamma_{\R}(s+\alpha_i)}^{-1}$
are precisely those of $\gamma_\infty(s,\pi,r)$ (including multiplicities).
In particular, $\{\alpha_1,\dots,\alpha_m\}$ is determined as a multiset, and hence
\eqref{eq: complconjalpha} follows from \eqref{eq: localgammaFE}.
\end{proof}

Assume that \eqref{eq: FE} is satisfied for some fixed $r$ and $\pi$
and take $\alpha_1,\dots,\alpha_m$ to be reduced (hence uniquely determined).
We set
\[
L_\infty^{\red}(s,\pi,r)=\prod_{i=1}^m\Gamma_{\R}(s+\alpha_i),
\]
so that
\[
\gamma_\infty(s,\pi,r)= c_\infty\frac{\overline{L_\infty(1-\bar{s},\pi,r)}}{L_\infty(s,\pi,r)}
\]
with $c_\infty = \pm C_\infty$.
In this case the archimedean conductor simplifies to
\begin{equation} \label{def: archcond}
\infcond_\infty(\pi,r)=\prod_{i=1}^m(1+\abs{\alpha_i-1/2}).
\end{equation}
For $p\in S_{\Q, f} (\pi)$ it follows from \eqref{eq: localgammaFE} that we can write in a unique fashion
\begin{equation} \label{eq: gammap}
\gamma_p(s,\pi,r)=c_p p^{(\frac12-s)\expcond_p(\pi,r)}P_p(p^{-s})/\bar{P}_p(p^{s-1}),
\end{equation}
where $c_p\in\C^*$, $\expcond_p(\pi,r)\in\Z$, and $P_p$ is a polynomial with $P_p(0)=1$ such that
no zeros $\alpha$ and $\beta$ of $P_p$ satisfy $\alpha \bar{\beta}=p^{-1}$. Here $\bar P$ is the polynomial obtained from $P$
by taking complex conjugates of the coefficients.
The degree of $P_p$ is the number of zeros (or, equivalently, poles) of $R_p$ in $\C^\times$ (counted with multiplicities), and
the integer $\expcond_p(\pi,r)$ is the difference between the order of $R_p$ at $X=0$ and the degree of $P_p$.
We have seen that $\expcond_p(\pi,r)$ and $P_p$ are uniquely determined by $\pi$ and $r$. Analogously to the case $p = \infty$,
the zeros of $P_p (p^{-s})$ are precisely the zeros of $\gamma_p(s,\pi,r)$.
Although we expect that $\expcond_p(\pi,r)\ge0$, we do not impose this condition at the outset.
We set
\[
L_p^{\red}(s,\pi,r)=P_p(p^{-s})^{-1}, \quad p\in S_{\Q, f} (\pi),
\]
and define the \emph{reduced completed $L$-function} \label{sec: redLfunction}
\[
L^{\red}(s,\pi,r)=\big( \prod_{p\in S_{\Q}(\pi)}L_p^{\red}(s,\pi,r) \big) L^{S(\pi)}(s,\pi,r),
\]
and the reduced epsilon factor
\[
\epsilon^{\red}(s,\pi,r)=c_\infty\prod_{p\in S_{\Q,f}(\pi)}c_p p^{(\frac12-s)\expcond_p(\pi,r)}= \arithcond(\pi,r)^{\frac12-s}
\prod_{p\in S_\Q (\pi)}c_p,
\]
where
\begin{equation} \label{def: arithcond}
\arithcond(\pi,r)=\prod_{p\in S_{\Q, f} (\pi)}p^{\expcond_p(\pi,r)}\in\Q_{>0}
\end{equation}
is the arithmetic, or finite, conductor.
Thus, we can rewrite \eqref{eq: FE} as
\[
L^{\red}(s,\pi,r)=\epsilon^{\red}(s,\pi,r)L^{\red}(1-s,\pi,r^\vee)
= \epsilon^{\red}(s,\pi,r) \overline{L^{\red}(1-\bar{s},\pi,r)}.
\]
We denote by $L^{\red, f}(s,\pi,r)= (\prod_{p\in S_{\Q,f}(\pi)}L_p^{\red}(s,\pi,r))L^{S(\pi)}(s,\pi,r)$ the
``finite'' part of $L^{\red}$.

\begin{remark} \label{rem: altproc}
In many cases there is an alternative procedure to define a completed $L$-function
\begin{equation} \label{OtherL}
L (s) = \big( \prod_{p\in S_{\Q}(\pi)} \tilde{L}_p(s) \big) L^{S(\pi)}(s,\pi,r),
\end{equation}
which may differ from $L^{\red}(s)$.
Here $\tilde{L}_\infty(s)=\prod_{i=1}^{\tilde{m}} \Gamma_{\R}(s+\tilde{\alpha}_i)$ with some complex numbers $\tilde{\alpha}_i$, and
$\tilde{L}_p(s)=\tilde{P}_p(p^{-s})^{-1}$ for $p\in S_{\Q, f} (\pi)$, where
$\tilde{P}_p$ are some polynomials satisfying $\tilde{P}_p (0) = 1$.
(For brevity, we say that factors $\tilde{L}_p(s)$, $p\in S_{\Q} (\pi)$, of this shape are Euler factors.)
The advantage in working with $L^{\red}(s)$ is that it is uniquely
determined by the partial $L$-function $L^{S(\pi)}$ (and hence by $\pi$ and $r$). Of course, it is only defined
if property \FE\ is known a priori.

In any case, as a consequence of the minimality of the local factors $L_p^{\red}(s)$, $p\in S_{\Q} (\pi)$,
the reduced $L$-function $L^{\red}(s)$ satisfies the following minimality property.
Suppose that a function $L(s)$ as in \eqref{OtherL} satisfies a functional equation of the form
\[
L (s) = c R^{\frac12 - s} L^\vee (1-s)
\]
for some $c\in\C^*$ and $R>0$, where
\[
L^\vee (s) = \big( \prod_{p\in S_{\Q}(\pi)} \tilde{L}^\vee_p(s) \big) L^{S(\pi)}(s,\pi,r^\vee),
\]
for some Euler factors $\tilde{L}^\vee_p(s)$, $p\in S_{\Q}(\pi)$.
Then for each $p\in S_{\Q,f}(\pi)$, the polynomial $\tilde{P}_p$ is divisible by $P_p$ and
$\tilde{L}_\infty(s)/L_\infty^{\red}(s,\pi,r)$ is a polynomial.
In particular, the quotient $L^{\red}(s) / L(s)$ is an entire function.
Moreover,
\[
\frac{\arithcond(\pi,r)}{R}
= \prod_{p\in S_{\Q,f}(\pi)}p^{\deg\tilde P_p-\deg P_p}
\]
is a positive integer.

Consider for example the case of $G = \GL (n)$ with the standard representation $r = \Std_n$.
If $\pi\in\Pi_{\cusp}(G(\A))$ then the completed $L$-function $L^{\GJ}(s,\pi,\Std_n)$ was defined and studied by Godement--Jacquet \cite{MR0342495}.
The Jacquet-Shalika bounds on the local parameters of generic representations imply that $L^{\GJ}(s,\pi,\Std_n)$ is reduced.
On the other hand, for any $m>1$ the $L$-function
\[
\prod_{j=1}^mL^{\GJ}(s+\frac{m+1}2-j,\pi,\Std_n),
\]
which is the completed $L$-function of a residual representation of $\GL(nm,\A)$, is not reduced.
%$\Gamma_{\R}(s+\frac12)\Gamma_{\R}(s-\frac12}\zeta(s+\frac12)\zeta(s-\frac12)$
%(which is the completed $L$-function of the identity representation of $\GL(2)$) is not reduced.
%for residual representations $\pi$, the reduced $L$-function will in general
%miss some Euler factors at the ramified primes (and have non-minimal archimedean factors).
\end{remark}

\subsection{}
Our next task is to apply the standard argument of Riemann and von Mangoldt for estimating the number of zeros of an $L$-function
(cf., e.g., \cite[\S 5.3]{MR2061214}).
For our purposes we will need a weak form of this argument which works under the following supplementary conditions.

\begin{definition} \label{DefinitionFEplus}
We say that $(G,r)$ satisfies property (FE+), if it satisfies property \FE\ and if
in addition the following holds.
\begin{enumerate}
\item There exists a polynomial $P(s)$, whose degree is bounded in terms of $(G,r)$ only, such that
$P(s)L^{S(\pi)}(s,\pi,r)$ is an entire function of finite order.
\item The degrees of the polynomials $P_p$, $p \in S_{\Q,f}$, in \eqref{eq: gammap}, as well as the degree of
the factor $\gamma_\infty$ are bounded in terms of $(G,r)$ only.
\item There exists a real number $\beta$ depending only on $(G,r)$, such that
the zeros of $P_p$, $p \in S_{\Q,f}$, have absolute value $\ge p^{-\beta}$ and
such that $\gamma_\infty (s,\pi,r)$ has no zeros in the half-plane $\Re s > \beta$
(i.e., the reduced parameters $\alpha_i$ satisfy $\Re \alpha_i \ge - \beta$).
\end{enumerate}

We say that $(G,r)$ has virtually property (FE+), if there exist representations $r_1$ and $r_2$ for which $(G,r_i)$ has property (FE+) and $r\oplus r_1 = r_2$.
\end{definition}

In practice, we will have $\deg\gamma_\infty=N[F:\Q]\ge \deg P_p$ for all $p\in S_{\Q,f}(\pi)$.
However, we do not demand it at the outset.
The most elusive condition in property (FE+) is in fact the first one --- the boundedness of the number of poles of $L^S(s,\pi,r)$ independently of $\pi$.

\begin{remark}
It is clear that $(G,r)$ satisfies (FE+) if and only if the same is true for $(G,r^\vee)$.
An analogous remark is applicable for all subsequent properties defined below.
\end{remark}

We can now state the desired estimate.
Recall the archimedean and arithmetic conductors defined in \eqref{def: archcond} and \eqref{def: arithcond}, respectively.
Here and throughout we write
\[
A\ll B
\]
to mean that $A$ is bounded by a constant multiple of $B$. The implied constant is allowed to depend only on the pair $(G,r)$.

\begin{proposition} \label{lem: logderbnd}
Suppose that $(G,r)$ satisfies property (FE+).
Then for any $\pi\in\Pi_{\disc}(G(\A))$ the meromorphic function
\[
m(s,\pi,r)=\epsilon^{\red}(1,\pi,r)L^{\red}(s,\pi,r)/L^{\red}(s+1,\pi,r)
\]
satisfies $\abs{m(\iii t,\pi,r)}=1$ for $t\in\R$ and
\[
\int_T^{T+1}\abs{m'(\iii t,\pi,r)}\ dt \ll
\log (\abs{T}+2) + \log \infcond_\infty(\pi,r) + \log ( \arithcond(\pi,r) \prod_{p\in S_{\Q,f} (\pi)} p),\ \ \ T\in\R.
\]
\end{proposition}

\begin{proof}
Although the argument is familiar, we will provide the details, since our assumptions are somewhat weaker than usual.
Let $\Lambda(s,\pi,r)=\arithcond(\pi,r)^{s/2}L^{\red}(s,\pi,r)$. Then
\begin{equation} \label{eq: FELambda}
\Lambda(s,\pi,r)=\epsilon^{\red}(\frac12,\pi,r)\overline{\Lambda(1-\overline{s},\pi,r)}.
\end{equation}
Note that $m(s,\pi,r)=\Lambda(s,\pi,r)/\overline{\Lambda(-\overline{s},\pi,r)}$ and hence $m(s,\pi,r)$ is holomorphic on the imaginary axis
and has absolute value one there. Moreover,
\[
\frac{m'(\iii t,\pi,r)}{m(\iii t,\pi,r)}=2\Re\frac{\Lambda'(\iii t,\pi,r)}{\Lambda(\iii t,\pi,r)}, \quad t\in\R.
\]

The function $\Lambda(s)=\Lambda(s,\pi,r)$ is a quotient of holomorphic functions of order one, and hence a meromorphic function of order one.
For instance, the argument of \cite[Proposition 1]{MR2230919} shows that there exist positive constants $c_1$, $c_2$, $c_3$, depending also on $\pi$, such that
\[
\abs{P(s)L^{S(\pi)}(s,\pi,r)}\le c_1(1+\abs{s})^{\max(c_2,-c_3\Re(s))},\ \ \ s\in\C.
\]
Therefore $\Lambda(s)$ admits a Hadamard factorization
\[
\Lambda(s)=e^{a+bs}s^{n(0)}\prod_{\rho\ne0}[(1-s/\rho)e^{s/\rho}]^{n(\rho)},
\]
where $a,b\in\C$, the product ranges over the zeros and poles of $\Lambda(s)$ other than $0$, and
$n(\rho)$ is the order of the function $\Lambda(s)$ at $s=\rho$ (possibly negative). Also,
\[
\sum_{\rho\ne0}\frac{\abs{n(\rho)}}{\abs{\rho}^{1+\epsilon}}<\infty
\]
for all $\epsilon>0$.
Thus,
\[
\frac{\Lambda'(s)}{\Lambda(s)}=b+\frac{n(0)}s+\sum_{\rho\ne0}n(\rho)\left(\frac1{s-\rho}+\frac1\rho\right),
\]
and hence
\[
\Re\frac{\Lambda'(s)}{\Lambda(s)}=\sum_\rho\Re\frac{n(\rho)}{s-\rho}+\Re b+\sum_{\rho\ne0}\Re\frac{n(\rho)}\rho,
\]
where the series of $\rho$ are absolutely convergent because of
\[
\sum_{\rho\ne0}\abs{\Re\frac{n(\rho)}\rho}=\sum_{\rho\ne0}\frac{n(\rho)\abs{\Re\rho}}{\abs{\rho}^2}\ll\sum_{\rho\ne0}\frac{n(\rho)}{\abs{\rho}^2}<\infty.
\]
Taking the logarithmic derivative of the functional equation \eqref{eq: FELambda}, we also get
\[
-\Re\frac{\Lambda'(s)}{\Lambda(s)}=\sum_\rho\Re\frac{n(\rho)}{1-s-\overline{\rho}}+\Re b+\sum_{\rho\ne0}\Re\frac{n(\rho)}\rho.
\]
Since $n(\rho)=n(1-\overline{\rho})$, we conclude that $\Re b+\sum_{\rho\ne0}\Re\frac{n(\rho)}\rho=0$, and therefore
%\[
%-\frac{\Lambda'(s)}{\Lambda(s)}=\overline{b}+\frac{n(0)}{1-s}+\sum_{\rho\ne0}n(\rho)\left(\frac1{1-s-\overline{\rho}}+\frac1{\overline\rho}\right).
%\]
%Thus,
%\[
%-2\Re b=n(0)\left(\frac 1s+\frac1{1-s}\right)+\sum_{\rho\ne0}n(\rho)\left(\frac1{s-\rho}+\frac1{1-s-\overline{\rho}}+\frac1{\rho}+\frac1{\overline\rho}\right).
%\]
%Since
%\[
%\frac1{s-\rho}+\frac1{1-s-\overline{\rho}}=\frac{1-2\Re\rho}{(s-\rho)(1-s-\overline{\rho})}
%\]
%and $\abs{\Re\rho}$ is bounded, $\abs{\frac1{s-\rho}+\frac1{1-s-\overline{\rho}}}\ll_s\abs{\rho}^{-2}$. Hence,
%the series $\sum_{\rho\ne0}n(\rho)\left(\frac1{s-\rho}+\frac1{1-s-\overline{\rho}}\right)$ and
%$\sum_{\rho\ne0}n(\rho)\left(\frac1{\rho}+\frac1{\overline\rho}\right)$ converge. Also, by \eqref{eq: FELambda}
%\[
%\sum_{\rho\ne 0,1}n(\rho)\left(\frac1{s-\rho}+\frac1{1-s-\overline{\rho}}\right)=0
%\]
%and therefore
%\[
%\Re b=-\sum_{\rho\ne0}\Re\frac{n(\rho)}{\rho}.
%\]
%Thus,
\[
\Re\frac{\Lambda'(s)}{\Lambda(s)}=\sum_\rho n(\rho)\Re\frac1{s-\rho}=\sum_\rho n(\rho)\frac{\Re (s-\rho)}{\abs{s-\rho}^2}.
\]

By the second and third condition of Definition \ref{DefinitionFEplus},
there exist an integer $m \ge 1$ and
a constant $A \ge 2$, depending only on $G$ and $r$, such that for all primes $p$ the Euler factor at $p$ of $L^{\red}(s,\pi,r)$ is a product of at most $m$ factors $(1- \alpha_{p,i} p^{-s})^{-1}$ with
$\abs{\alpha_{p,i}} \le p^{A-2}$ for all $p$ and $i$, and such that the poles of the factor at infinity $L_\infty$ lie in the half-plane $\Re s \le A-2$.
(In fact, we may take A = $\max (\alpha,\beta) + 2$.)
Therefore, the Euler product
$L^{\red, f}(s,\pi,r)$
is absolutely convergent for $\Re s > A-1$, and because of the condition on $L_\infty$ and the functional equation \eqref{eq: FELambda}, all zeros and poles $\rho$ of
$\Lambda(s)$ lie in the strip $2 - A \le \Re s \le A-1$.

By definition, we have
\[
\frac{\Lambda'(s)}{\Lambda(s)}=\frac12\log\arithcond(\pi,r)+\frac{(L^{\red, f})'(s,\pi,r)}{L^{\red, f}(s,\pi,r)}
+\frac{L'_\infty (s,\pi,r)}{L_\infty (s,\pi,r)}.
\]
The absolute convergence of the Euler product for $\Re s > A-1$ implies that
\[
\abs{\frac{(L^{\red, f})'(s,\pi,r)}{L^{\red, f}(s,\pi,r)}}
\le - m \frac{\zeta'(s - A + 2)}{\zeta (s - A + 2)}
\]
in this half-plane.
Using the fact that $\mathbf{\Gamma}'(s)/\mathbf{\Gamma}(s)=\log s+O(1/\abs{s})$ for $\Re s\ge1$, we get that
\[
\abs{\Re\frac{\Lambda'(s)}{\Lambda(s)}}\le\abs{\frac{\Lambda'(s)}{\Lambda(s)}}\le\frac12\log\arithcond(\pi,r)+ \frac12 \log\infcond_\infty(\pi,r)+
\frac{m}2 \log\abs{s} + c
\]
for $\Re s \ge A$, where $c$ depends only on $G$ and $r$.
Taking $\Re s = A$, we conclude that
\[
\sum_\rho n(\rho)\frac{A-\Re\rho}{\abs{A+\iii T-\rho}^2}\le\frac12\log\arithcond(\pi,r) + \frac12 \log\infcond_\infty(\pi,r)+ \frac{m}2 \log(1+\abs{T})+c.
\]
On the other hand,
because of the first condition of Definition \ref{DefinitionFEplus}, up to finitely many exceptions
the poles of $\Lambda (s)$ can only arise from the poles of the Euler factors $L_p(s,\pi,r) = P_p (p^{-s})$, $p \in S_{\Q, f} (\pi)$.
Therefore, using the identity
$\sum_{n\in\Z}(1+(n/x)^2)^{-1}=\pi x\coth(\pi x)$,
we have
\[
\sum_{\rho: n (\rho) < 0} \abs{n(\rho)}\frac{A-\Re\rho}{\abs{A+\iii T-\rho}^2}\le c+\frac m2\sum_{p\in S_{\Q, f} (\pi)}\frac{1+p^{-1}}{1-p^{-1}}\log p,
\]
where $c$ again depends only on $r$.
Hence
\begin{equation} \label{eq: bnd on zeros}
\sum_\rho\frac{\abs{n(\rho)}}{1+(T-\Im\rho)^2}\ll\\ \log\arithcond(\pi,r)+\sum_{p \in S_{\Q,f}(\pi)}\log p+
\log \infcond_\infty(\pi,r) +\log(1+\abs{T})+1,
\end{equation}
and in particular
\begin{equation} \label{eq: bnd on zeros2}
\sum_{\rho:\abs{\Im\rho-T}<2}\abs{n(\rho)}\ll
\log\arithcond(\pi,r) + \sum_{p\in S_{\Q,f} (\pi)}\log p+
\log \infcond_\infty(\pi,r) + \log(1+\abs{T})+1.
\end{equation}

We can now estimate $\int_T^{T+1}\abs{m'(\iii t,\pi,r)}\ dt=2\int_T^{T+1}\abs{\Re\frac{\Lambda'(\iii t)}{\Lambda(\iii t)}}\ dt$
by writing
\[
\Re\frac{\Lambda'(\iii t)}{\Lambda(\iii t)}=\sum_\rho n(\rho)\frac{\Re\rho}{\abs{\iii t-\rho}^2},
\]
and splitting the sum into two parts according to whether $\abs{\Im\rho-T}\ge2$ or $\abs{\Im\rho-T}<2$. For the first sum we have
\[
\abs{\sum_{\rho:\abs{\Im\rho-T}\ge2}n(\rho)\frac{\Re\rho}{\abs{\iii t-\rho}^2}}
\le 2A\sum_{\rho:\abs{\Im\rho-T}\ge2}\frac{\abs{n(\rho)}}{1+(T-\Im\rho)^2}
\]
for any $t\in [T,T+1]$ and we use \eqref{eq: bnd on zeros}. For the second sum we use the fact that for $\rho=\beta+\iii\gamma$
with $\beta\ne0$ we have
\[
\int_T^{T+1}\frac{\abs{\Re\rho}}{\abs{\iii t-\rho}^2} dt=\int_T^{T+1}\frac{\abs{\beta}}{\beta^2+(t-\gamma)^2}\ dt=
\int_{(T-\gamma)/\abs{\beta}}^{(T+1-\gamma)/\abs{\beta}}\frac1{1+t^2}\ dt<\pi.
\]
Thus,
\[
\int_T^{T+1}\abs{\sum_{\rho:\abs{\Im\rho-T}<2}n(\rho)\frac{\Re\rho}{\abs{\iii t-\rho}^2}}\ dt\le
\pi\sum_{\rho:\abs{\Im\rho-T}<2}\abs{n(\rho)},
\]
which is bounded by \eqref{eq: bnd on zeros2}.
All in all we get the required estimate.
\end{proof}

%\begin{definition}
%We say that $r$ is \WNICE\ if we can write
%\end{definition}
%Note that if $r$ is \WNICE\ then the partial $L$-functions admit meromorphic continuation and functional equation
%but they are not required (a priori) to have finitely many poles.

\begin{remark} \label{rem: weaknice is enough}
It is clear that in Proposition \ref{lem: logderbnd} it is in fact sufficient to require that
$(G,r)$ has virtually property (FE+).
Similarly, it suffices to know that the number of poles of $L^{S(\pi)}(s,\pi,r)$ is
$\ll \log \infcond_\infty(\pi,r) + \log ( \arithcond(\pi,r) \prod_{p\in S_{\Q,f} (\pi)} p)$.
% we can replace the condition that $r$ is \NICE\ by requiring that $r=r_1-r_2$ (as virtual representations)
%where $r_1$ and $r_2$ are \NICE. (Thus for instance, the condition that %$L^{S(\pi)}(s,\pi,r)$ has finitely many poles is not absolutely necessary.)
\end{remark}

In order to make the connection to asymptotic problems, we need to control $\arithcond(\pi,r)$ (resp., $\infcond_\infty(\pi,r)$)
in terms of the level of $\pi$ (resp., the size of the infinitesimal character of $\pi_\infty$).
Fix once and for all a faithful $F$-rational representation $\rho: G \to \GL (V)$ and an $\mathfrak{o}_F$-lattice $\Lambda$ in the representation space $V$.
The stabilizer of $\hat{\Lambda} = \hat{\mathfrak{o}}_F \otimes \Lambda \subset \A_{\fin} \otimes V$ in $G(\A_{\fin})$ is an open compact subgroup $\K_{\fin}$,
and any maximal compact subgroup of $G(\A_{\fin})$ can be realized this way.
%(Since the maximal compact subgroups of $\GL (\A_{\fin} \otimes V)$ are precisely the stabilizers of lattices, it is easy to see that such a lattice exists.)
For any non-zero ideal $\nnn$ of $\mathfrak{o}_F$ let
\[
\K (\nnn) = \K_G(\nnn)=\{ g \in G (\A_{\fin}) \, : \, \rho (g) v \equiv v \pmod{\nnn \hat{\Lambda}}, \quad v \in \hat{\Lambda} \}
\]
be the principal congruence subgroup of level $\nnn$, a factorizable normal open subgroup of $\K_{\fin}$.
The groups $\K (\nnn)$ form a neighborhood base of the identity element in $G(\A_{\fin})$.
%and $\K(\nnn_1+\nnn_2)=\K(\nnn_1)\K(\nnn_2)$ for any ideals $\nnn_1$, $\nnn_2$ of $\mathfrak{o}_F$.
We denote by $\inorm (\nnn) = [ \mathfrak{o}_F : \nnn]$ the ideal norm of $\nnn$.
We define the level of an admissible representation $\pi$ of $G(\A)$ by
$\level (\pi)=\inorm(\nnn)$, where $\nnn$ is the largest ideal such that
$\pi^{\K(\nnn)}\ne0$. Analogously, for any finite place $v$ of $F$ we define the level $\level_v (\pi_v)$ of a smooth representation $\pi_v$ of $G(F_v)$.
Thus, $\level(\pi)=\prod_v\level_v(\pi_v)$ where $v$ ranges over the finite places of $F$ and almost all of the factors are $1$.
Note that there exists an integer $n$, depending only on $G$, such that for any $\pi\in\Pi_{\disc}(G(\A))$ we have:
\begin{equation} \label{eq: levelSpi}
\text{$p$ divides $\level(\pi)$ for any rational prime $p \in S_{\Q,f}(\pi)$ coprime to $n$.}
\end{equation}

We fix a maximal compact subgroup $\K_\infty$ of $G(F_\infty)$ and set $\K = \K_\infty \K_{\fin}$.

Fix a Cartan subalgebra $\Crtn_{\C}$ of the complexified Lie algebra $(\Lieg_\infty)_{\C}$ of $G(F_\infty)$
and a Euclidean norm $\norm{\cdot}_{\C}$ on $\Crtn_{\C}^*$ which is invariant under the Weyl group $W((\Lieg_\infty)_{\C},\Crtn_{\C})$.
For any (smooth) non-zero representation $\pi$ of $G(F_\infty)$, which has an infinitesimal character $\chi_\pi$, we define
\[
\param(\pi)=1+\norm{\chi_\pi}^2,
\]
where we view $\chi_\pi$ as an orbit in $\Crtn_{\C}^*$ under $W((\Lieg_\infty)_{\C},\Crtn_{\C})$.
\begin{remark} \label{rem: aboutparam}
In \cite[\S 5.1]{MR3352530} the slightly different invariant
\[
\param'(\pi)=1+\lambda_\pi^2+\min\norm{\tau}^2
\]
is used (and denoted there by $\param(\pi)$),
where $\lambda_\pi$ is the eigenvalue of the Casimir operator of $G (F_\infty)$,
$\tau$ ranges over the $\K_\infty$-types of $\pi$ and the norm $\norm{\tau}$ is the one defined by Vogan
(cf. \cite[\S 2.2]{MR759263}).
By a standard argument (cf. \cite[\S 6.5--6.6]{MR632407}), there exist constants $c_1,c_2>0$, depending only on $G$, such that
\[
c_1\param'(\pi)<\param(\pi)^2<c_2\param'(\pi).
\]
Therefore, for all practical purposes there is no difference between $\param(\pi)$ and $\param'(\pi)$.
\end{remark}

We thus make the following definition.

\begin{definition}
Suppose that the pair $(G,r)$ satisfies property (FE+).
We say that the conductor condition \CC\ is satisfied, if
there exists $c > 0$ (depending only on $G$ and $r$), such that
\begin{equation} \label{eq: arithcndbnd}
\arithcond(\pi,r)\ll\level (\pi)^c
\end{equation}
and
\begin{equation} \label{eq: archimedean}
\infcond_\infty(\pi,r) \ll \param (\pi_\infty)^c
\end{equation}
for all $\pi\in\Pi_{\disc}(G(\A))$.
\end{definition}

Condition \CC\ is certainly expected to hold in general (cf.~\cite{1311.1606}), although it is not clear whether, strictly speaking, its non-archimedean part \eqref{eq: arithcndbnd}
formally follows from Langlands's principle of functoriality or the local Langlands conjecture for $G$.
For the archimedean part \eqref{eq: archimedean} see Remark \ref{rem: archcc} below.
%is a consequence of the conjectural description of the factor $\gamma_\infty$ in terms of the Langlands classification (cf. Lemma \ref{LemmaArchimedean} below).

%We emphasize again that the first condition of \CC\ implicitly depends on the choice of local factors $\gamma_v(s,\pi,r)$.

\subsection{} \label{sec: psivconv}
The standard paradigm for proving \FE\ (e.g., through the Langlands-Shahidi method or an integral representation, cf.~\cite{MR951897})
goes by defining local factors $\gamma_v^{\gnr}(s,\pi,r)$, that satisfy certain properties, most importantly
\begin{equation} \tag{FE'} \label{eq: global FE}
L^S(s,\pi,r)=\big(\prod_{v\in S}\gamma_v^{\gnr}(s,\pi,r)\big)L^S(1-s,\pi,r^\vee)
\end{equation}
for any finite set $S\supset S(\pi)$.
(In each particular case the superscript $^\gnr$ will be replaced by an appropriate acronym.)
The local factors $\gamma_v^{\gnr}(s,\pi,r)$ normally depend on a choice of a character $\psi_v$ of $F_v$.
We will suppress this choice by taking $\psi_F=\psi_\Q\circ\Tr_{F/\Q}$
where $\psi_\Q=\prod_p \psi_{\Q_p}$ is the standard character of $\Q\bs\A_\Q$ (characterized by $\psi_{\R}=e^{2\pi\iii\cdot}$)
and writing $\psi_F=\otimes\psi_v$, i.e., $\psi_v=\psi_{F_v}=\psi_{\Q_p}\circ\Tr_{F_v/\Q_p}$, if $v$ lies above $p$.
If we want to emphasize that $\gamma_v^{\gnr}(s,\pi,r)$ depends only on $\pi_v$ (as it is in the usual paradigm) we will write it
as $\gamma_v^{\gnr}(s,\pi_v,r)$. However, we will not make it a part of our requirements.
This flexibility will be useful in sections \ref{sec: goodLexamples} and \ref{SectionClassical}, where we will use functoriality to study analytic properties of $L$-functions.

In addition to \eqref{eq: global FE} we will impose the following conditions
on $\gamma_v^{\gnr}(s,\pi,r)$ (which for simplicity we denote by \GF).
Here, $n$ is a positive integer and $\beta$ a real number which depend only on $(G,r)$. (In practice $n=N=\deg r$.)
\begin{enumerate}
\item For every finite $v$, the function $\gamma_v^{\gnr}(s,\pi,r)$ is a rational function in $q_v^{-s}$, and
\begin{equation} \label{eq: gammadot}
\gamma_v^{\gnr}(s,\pi,r)= \epsilon_v^{\gnr}(s,\pi,r)\frac{L_v^{\gnr}(1-s,\pi,r^\vee)}{L_v^{\gnr}(s,\pi,r)},
\end{equation}
where, denoting by $\Disc_v$ the ideal norm of the different of $F_v$ (i.e., the conductor of $\psi_v$ as above),
\[
\epsilon_v^{\gnr}(s,\pi,r) = c_v q_v^{\expcond_v^{\gnr}(\pi,r)(\frac12-s)}\Disc_v^{N(\frac12-s)},\ \
L_v^{\gnr}(s,\pi,r) = P_v(q_v^{-s})^{-1},\ \ L_v^{\gnr}(s,\pi,r^\vee) = Q_v(q_v^{-s})^{-1}
\]
for some $\expcond_v^{\gnr}(\pi,r)\in\Z$ and polynomials $P_v$ and $Q_v$ of degree $\le n$ satisfying $P_v (0) = Q_v (0) = 1$.
\item For $v\in S_\infty$ we have
\begin{equation} \label{eq: gammaarch}
\gamma_v^{\gnr}(s,\pi,r)=c_v\prod_{i=1}^{m_v}\frac{\Gamma_{F_v}(1-s+\alpha_{v,i}^\vee)}{\Gamma_{F_v}(s+\alpha_{v,i})}
\end{equation}
for some $c_v\in\C^*$, $\alpha_{v,1},\dots,\alpha_{v,m_v},\alpha_{v,1}^\vee,\dots,\alpha_{v,m_v}^\vee\in\C$, where $m_v \le n$.
%%%\item $\gamma_v^{\gnr}(s,\pi,r) \gamma_v^{\gnr}(1-s,\pi,r^\vee)=\text{const.}$ for all $v\in S (\pi)$.
\item For all $v\in S(\pi)$, the function $\gamma_v^{\gnr}(s,\pi,r)$ has no zeros for $\Re s>\beta$.
\end{enumerate}

Let us make a few comments about these conditions.

First, as before, we do not impose that $\expcond_v^{\gnr}(\pi,r)\ge0$ for $v\in S_f(\pi)$
(although in practice this will always be the case).
Also, for $v$ finite we allow $P_v (X)$ and $Q_v (q_v^{-1} X^{-1})$ to have common zeros.
Thus, $\expcond_v^{\gnr}(\pi,r)$ does not depend only on the function $\gamma_v^{\gnr}(s,\pi,r)$ itself.
However, this ambiguity is immaterial for our purpose since different presentations \eqref{eq: gammadot}
give rise to values of $\expcond_v^{\gnr}(\pi,r)$ which differ by an integer of absolute value $\le n$.

Similarly, for $v\in S_\infty$ we allow the numerator and denominator in the expression \eqref{eq: gammaarch} to have common poles.
We set
\[
\infcond_v^{\gnr}(\pi,r)=\big(\prod_{i=1}^{m_v}(1+\abs{\alpha_{v,i}-1/2})(1+\abs{\alpha_{v,i}^\vee-1/2})\big)^{\frac12}.
\]
As in the proof of Lemma \ref{lem: elem}, $\infcond_v^{\gnr}(\pi,r)$ depends only on the
factor $\gamma_v^{\gnr}(s,\pi,r)$ and not on the choice of $\alpha_{v,1},\dots,\alpha_{v,m_v},\alpha_{v,1}^\vee,\dots,\alpha_{v,m_v}^\vee$.
(In practice we always have $m_v=N$.)

Note that the functional equation \eqref{eq: global FE} implies that for all $v\notin S(\pi)$ we have
\[
\gamma_v^{\gnr}(s,\pi,r)=\frac{L_v(1-s,\pi_v,r^\vee)}{L_v(s,\pi_v,r)}.
\]
Finally, it implies \FE\ with
\[
\gamma_p(s,\pi,r)=\prod_{v\in S(\pi):v|p}\gamma_v^{\gnr}(s,\pi,r)
\]
for each $p\in S_{\Q} (\pi)$, and the second and third conditions of property (FE+) are satisfied. Regarding condition (CC),
by Remark \ref{rem: altproc} the quotient
\[
\frac{\arithcond(\pi,r)}{\prod_{v\in S_f(\pi)}q_v^{\expcond_v^{\gnr}(\pi,r)}}
\]
is a positive integer which divides $\prod_{p\in S_{\Q, f} (\pi)} p^{n[F:\Q]}$.
%we have the divisibilities
%\[
%\prod_{v\in S_f(\pi)}q_v^{\expcond_v^{\gnr}(\pi,r)} \, \Big| \, \arithcond(\pi,r)
%\, \Big| \, \prod_{v\in S_f(\pi)}q_v^{\expcond_v^{\gnr}(\pi,r)}
%( \prod_{p\in S_{\Q, f} (\pi)} p)^{n[F:\Q]}.
%\]
Therefore, by \eqref{eq: levelSpi} the non-archimedean part of condition (CC) is implied by the condition
\begin{equation} \tag{CC'} \label{eq: locndbnd}
\expcond_v^{\gnr}(\pi,r)\ll 1+\log_{q_v}\level_v(\pi_v),\ \ \ v\in S_f(\pi).
\end{equation}
Note that this property (in contrast to the value of $\expcond_v^{\gnr}(\pi,r)$ itself) depends only on $\gamma_v^{\gnr}(s,\pi,r)$.
The archimedean part of condition (CC) is implied by the condition
\begin{equation} \tag{AF} \label{eq: archbnd}
\infcond_v^{\gnr}(\pi,r)\ll\param(\pi_v)^c,\ \ \ v\in S_\infty,
\end{equation}
where $c$ depends only on $(G,r)$.

\begin{remark} \label{rem: archcc}
In many cases one knows that the archimedean factors are compatible with the Langlands classification in the following sense.
For $v\in S_\infty$, let $W_v$ be the Weil group of $F_v$ and let $\phi_v:W_v\rightarrow\LG$ be the Langlands parameter attached to $\pi_v$.
Then
\begin{equation} \tag{CL} \label{eq: comparch}
\gamma_v^{\gnr}(s,\pi,r)=\gamma_v(s,r\circ\phi_v,\psi_v):=\epsilon_v(s,r\circ\phi_v,\psi_v)L_v(1-s,r^\vee \circ\phi_v)/L_v(s,r\circ\phi_v),
\end{equation}
where the $L$- and $\epsilon$-factors on the right-hand side are as in \cite[\S 3]{MR546607}.
It is easy to see that condition \eqref{eq: comparch} for all $v\in S_\infty$ implies \eqref{eq: archbnd} (with $c=N/2$).
\end{remark}

We sum up the discussion as follows.

\begin{corollary} \label{cor: enufsomegamma}
For a given pair $(G,r)$, suppose that for all $\pi \in \Pi_{\disc}(G(\A))$ the following conditions are satisfied.
\begin{enumerate}
\item \label{item: one} There exists a polynomial $P(s)$, whose degree is bounded in terms of $(G,r)$ only, such that
$P(s)L^{S(\pi)}(s,\pi,r)$ extends to an entire function of finite order.
\item \label{item: two} There exist local factors $\gamma_v^{\gnr}(s,\pi,r)$ (for all places $v$ of $F$) satisfying \eqref{eq: global FE}, \GF,
\eqref{eq: archbnd} and (CC').
\end{enumerate}
Then $(G,r)$ satisfies properties (FE+) and (CC).
\end{corollary}

The following observation will be useful.
\begin{lemma} \label{lem: cuspenuf}
Suppose that for all triplets $(M,\pi,r')$, where $M$ is a Levi subgroup of $G$ defined over $F$ (including $G$ itself),
$\pi\in\Pi_{\cusp}(M(\A))$ is a \emph{cuspidal} representation of $M(\A)$, and $r'$ is an irreducible constituent of $r\rest_{\LM}$,
conditions \eqref{item: one} and \eqref{item: two} of Corollary \ref{cor: enufsomegamma} are satisfied (with $M$ instead of $G$ and $r'$ instead of $r$).
Then %the assumptions of Corollary \ref{cor: enufsomegamma} hold, and hence
$(G,r)$ satisfies properties (FE+) and (CC).
\end{lemma}

\begin{proof}
Let $\pi\in\Pi_{\disc}(G(\A))$. Then there exist $\sigma\in\Pi_{\cusp}(M(\A))$ and $\lambda\in\aaa_{M,\C}^*$ such that
$\pi$ is a subquotient of the representation $I(\sigma,\lambda)$ parabolically induced from the twist of $\sigma$ by the character
of $M(\A)/M(\A)^1$ determined by $\lambda$ \cite[Supplement]{MR546598}.
We have $S(\sigma)\subset S(\pi)$, $\level(\sigma)\ll\level(\pi)$ and $\param(\sigma_\infty)+\norm{\lambda}^2\ll\param(\pi_\infty)$.
Also, by the unitarity of $\pi$, $\norm{\Re\lambda}$ is bounded in terms of $G$ only.
Decompose $r\rest_{\LM}=\oplus_{i=1}^kr_i$ according to the central character, and let $\beta_i^\vee$ be the element
of $\aaa_M$ corresponding to the central character of $r_i$. Then
\[
L^{S(\pi)}(s,\pi,r)=\prod_{i=1}^kL^{S(\pi)}(s+\sprod{\lambda}{\beta_i^\vee},\sigma,r_i).
\]
We take
\[
\gamma_v^{\gnr}(s,\pi,r)=\prod_{i=1}^k\gamma_v^{\gnr}(s+\sprod{\lambda}{\beta_i^\vee},\sigma,r_i).
\]
Then properties \eqref{item: one} and \eqref{item: two} for $(G,\pi,r)$ immediately follow from the corresponding properties
of $(M,\sigma,r_i)$, $i=1,\dots,k$ (cf.~Remark \ref{rem: altproc}).
\end{proof}

%In the next section we will give examples of pairs $(G,r)$ which are \NICE\ and satisfy \CC.

%$L_\infty(s,\pi_\infty,r_j):=\prod_{v|\infty}L_v(s,\pi_v,r_j)$

\subsection{} \label{sec: relparams}
At this stage it will be useful to introduce a slight refinement of the notion of level.
For the rest of this section let $G'$ be a closed connected normal subgroup of $G$ and $p:H\rightarrow G'$ a finite covering of $G'$
(with $H$ connected), both defined over $F$.
Note that $p(H(F))$ is a normal subgroup of $G'(F)$ and that $G'(F)/p(H(F))$ is an abelian group of finite exponent (bounded by the size of the kernel of $p$), and
similarly for $p(H(\A))\subset G'(\A)$.

For any $\pi\in\Pi_{\disc}(G(\A))$ we write $\level(\pi;p)=\inorm(\nnn)$, where $\nnn$ is the largest ideal such that
$\pi^{\K(\nnn)\cap p(H(\A))}\ne0$.
(Note that the same notion was considered in \cite[\S 5.1]{MR3352530}, where the notation $\level(\pi; p(H(\A)))$ was used.)
Analogously, we define $\level_v(\pi_v;p)$ for a smooth representation $\pi_v$ of $G(F_v)$.
We also set
\begin{equation} \label{def: relparam}
\param(\pi_\infty;p) =\param^{G'}(\pi_\infty\rest_{G'(F_\infty)}),
\end{equation}
where on the right-hand side $\param$ is taken with respect to $G'$.
Alternatively, $\param(\pi_\infty;p) = 1 + \norm{\chi_{\pi_\infty;p}}^2$,
where $\chi_{\pi_\infty;p}$
is the projection of $\chi_{\pi_\infty}$ to $(\Crtn_{\C}\cap\Lie(G'(F_\infty))_{\C})^*$.

\begin{lemma} \label{lem: restH}
There exists an integer $N_1$, depending only on $p$ and $G$, such that
for any $\pi\in\Pi_{\disc}(G(\A))$ there exists $\sigma\in\Pi_{\disc}(G'(\A))$ and a character $\chi$ of $G'(\A)$
trivial on $G'(F)p(H(\A))$ such that $\level(\sigma)$ divides $N_1\level(\pi;p)$ and $\sigma\chi$ is a subrepresentation of $\pi\rest_{G'(\A)}$.
In particular, $\param^{G'}(\sigma)=\param(\pi_\infty;p)$.
\end{lemma}

%Indeed, for the first part it is enough to note that if $\sigma\in\Pi_{\disc}(G'(\A))$ and $\chi$ is a character of $G'(\A)/G'(F)p(H(\A))$ then
%$L^S(s,\sigma\chi,r')=L^S(s,\sigma,r')$ and if $\sigma$ occurs in the restriction of $\pi\in\Pi_{\disc}(G(\A))$
%then $L^S(s,\pi,r)=L^S(s,\sigma,r')$.

\begin{proof}
We first reduce the lemma to the case $G'=G$. Namely, we show that there exists an integer $N_2$, depending only on $G$ and $G'$,
such that for any $\pi\in\Pi_{\disc}(G(\A))$ there exists a subrepresentation
$\sigma\in\Pi_{\disc}(G'(\A))$ of $\pi |_{G'(\A)}$ with
$\level (\sigma;p)$ dividing $N_2 \level (\pi;p)$.

Let $C$ be the centralizer of $G'$ in $G$.
Then $CG'=G$, and therefore $G'(F_\infty) C(F_\infty)$ has finite index in $G(F_\infty)$.
Since the maximal compact subgroup $\K_\infty$ meets every connected component of $G(F_\infty)$, we have
$G(F_\infty) = G' (F_\infty) C(F_\infty) \K_\infty$.
Combining this fact with the well-known finiteness of the class number of $G$ \cite[Theorem 8.1]{MR1278263}, we see that the coset space
\[
X = G(F) \bs G (\A) / G' (\A) C(F_\infty) \K
\]
is finite. Fix a set of representatives $\{g_1, \ldots, g_r\}$ for the classes of $X$.

Let $\pi\in\Pi_{\disc}(G(\A))$.
Note that since $CG'=G$, $G(\A)/C(\A)G'(\A)$ is compact, and therefore, the restriction $\pi\rest_{G'(\A)}$ decomposes
into a direct sum of irreducible representations.
Let $\phi$ be an automorphic form on $G(\A)$ in the isotypic space of $\pi$,
which is right-invariant under the group $\K(\nnn) \cap p(H(\A_{\fin}))$.
Then there exist $i=1,\dots,r$, $k \in \K_M$ and $c \in C(F_\infty)$, such that the function $\tilde\phi:=\phi (\cdot c g_i k)$ does not vanish on
$G' (F) \bs G' (\A)$.
The function $\tilde\phi$ is clearly invariant under $g_i \K(\nnn) g_i^{-1} \cap p(H(\A_{\fin}))$.
Decomposing the span of $\tilde\phi\rest_{G'(\A)}$ under $G'(\A)$ into irreducibles, we obtain an irreducible subrepresentation
$\sigma\in\Pi_{\disc}(G'(\A))$ of $\pi |_{G'(\A)}$ such that
$\level (\sigma;p)$ divides $N_2 \level (\pi;p)$, where the integer $N_2$ depends on $G$ only, as required.

So assume from now on that $G'=G$.
It remains to show that there exists an integer $N_3$, depending only on $p$, such that for any $\sigma\in\Pi_{\disc}(G(\A))$ there exists a character
$\chi: G(\A)/ G(F) p(H(\A))\to \C^\times$ such that
$\level (\sigma\chi^{-1})$ divides $N_3 \level (\sigma;p)$.
Recall from \cite[Lemma 5.1]{MR3352530} that
\[
\level_v (\sigma_v ;p)= \min_{\chi_v} \level_v (\sigma_v \chi_v^{-1}),
\]
where $\chi_v$ ranges over the characters of $G(F_v)/ p(H(F_v))$.
Since $G(F_v)/ p(H(F_v))$ is finite for all $v$, it suffices to show the following assertion.
There exists a finite set $S_0$ of finite places of $F$, depending on $p$ only, such that for any finite set $S$ of finite places of $F$ outside $S_0$
and a family of characters $\tilde \chi_v$ of $G(F_v)/ p(H(F_v))$ for $v \in S$, there exists a global character
$\chi: G(\A) / G(F)p(H(\A))\to \C^\times$, unramified outside $S \cup S_0$, such that
$\chi_v \tilde\chi_v^{-1}$ is unramified for all $v \in S$. (We use the fact that
$\level_v(\pi_v\chi_v)\ll \level_v(\pi_v)$ for any character $\chi_v$ of $G(F_v)/p(H(F_v))$, $v\in S_0$.)

For convenience we write $\tilde\K=G(F_\infty)\K_{\fin}$.
We first show that the group
\[
(G(F)\cap p(H(\A))\tilde\K)/p(H(F))
\]
is finitely generated.
Indeed, choose representatives $x_1,\dots,x_k$ for the finite double coset
space $H(F)\bs H(\A)/p^{-1}(\tilde \K)$, and let $y_i=p(x_i)$.
Then $G(F)\cap p(H(\A))\tilde\K$ is the union over $i=1,\dots,k$ of
\[
G(F)\cap p(H(F))y_i\tilde\K=p(H(F))(G(F)\cap y_i\tilde\K).
\]
If $\delta_i\in G(F)\cap y_i\tilde\K$, then $G(F)\cap y_i\tilde\K=(G(F)\cap \delta_i\tilde\K \delta_i^{-1})\delta_i$,
which is a coset of a finitely generated group. Our claim follows.

Let $\Gamma$ be the image of $G(F)\cap p(H(\A))\tilde\K$ in $G(\A_{\fin}) / p(H(\A_{\fin}))$.
Since the latter is abelian of finite exponent, $\Gamma$ is necessarily finite, and therefore projects injectively into
$G (F_{S_0}) / p(H(F_{S_0}))$ for a suitable finite set $S_0$ depending only on $p$.
It is therefore possible to extend the character $\prod_{v \in S} \tilde\chi_v$ of $\K_S$ to a character of
$\tilde\K / (G(F) p(H(\A)) \cap\tilde\K)$, which is trivial on $\K_v$ for any $v\notin S \cup S_0$.
Extending this character to $G(\A) / G(F)  p(H(\A))$ we obtain the desired $\chi$.

The last assertion of the lemma is clear, since $\chi$ is trivial on the connected component of the identity of $G(F_\infty)$.
\end{proof}

%\begin{corollary}
%Let $p$ and $G'$ be as before and let $^L\iota:\,^LG\rightarrow\,^LG'$ be the homormophism corresponding
%to the embedding $\iota:G'\rightarrow G$.
%Suppose that $r'$ is a representation of $^LG'$ which factors through $^Lp:\,^LG'\rightarrow\,^LH$
%and let $r=r'\circ\, ^L\iota$.
%\begin{enumerate}
%\item Assume that $(G',r')$ satisfies (FE+) and (CC). Then $(G,r)$ satisfies (FE+) and (CC+).
%\item Assume that $G'$ contains the derived group of $G$ and $(G,r)$ satisfies (FE+) and (CC).
%Then $(G',r')$ satisfies (FE+) and (CC).
%\end{enumerate}
%\end{corollary}

The following observation will also be useful.
\begin{lemma} \label{lem: norsgr}
Assume that $G'$ contains the derived group of $G$ and let $^L\iota:\,^LG\rightarrow\,^LG'$ be the corresponding homomorphism of $L$-groups.
Let $r'$ be a representation of $^LG'$ and let $r=r'\circ\, ^L\iota$.
Assume that $(G,r)$ satisfies (FE+) and (CC). Then $(G',r')$ satisfies (FE+) and (CC).
\end{lemma}

We first need the following standard result.
\begin{lemma} \label{lem: torus}
Let $T$ be a torus over $F$. Then there exists a compact subset $C$ of the Pontryagin dual $T(F_\infty)^D =
\operatorname{Hom} (T(F_\infty), \C^1)$ of $T(F_\infty)$, such that
for any character $\chi$ of $T(\A)$ there exists a character $\tilde\chi$ of $T(F)\bs T(\A)$ such that
$\tilde\chi\chi^{-1}$ is unramified at all finite places and $\tilde\chi_\infty\chi^{-1}_\infty\in C$.
\end{lemma}

\begin{proof}
Indeed, let $T(F_\infty)^1=T(F_\infty)\cap T(\A)^1$. Then $X:=T(F)T(F_\infty)^1\prod_{v\text{ finite}}T(\OO_v)$ is a closed subgroup of finite index
of $T(\A)^1$ and the group $\Gamma:=T(F)\cap T(F_\infty)^1\prod_{v\text{ finite}}T(\OO_v)$ is a lattice in $T(F_\infty)^1$.
Therefore, there exists a compact subset $C$ of $(T(F_\infty)^1)^D$, such that its image under the restriction map
\[
(T(F_\infty)^1)^D\rightarrow \Gamma^D
\]
is onto. Thus, there exists a character $\tilde\chi$ of $X$, trivial on $T(F)$, whose restriction to $T(F_\infty)^1$
is in $\chi_\infty C$ and whose restriction to $\prod_{v\text{ finite}}T(\OO_v)$ is $\chi$. Extending $\tilde\chi$ arbitrarily to
$T(\A)^1$ and setting $\tilde\chi |_{A_T} = \chi |_{A_T}$, we obtain the assertion.
\end{proof}

\begin{proof}[Proof of Lemma \ref{lem: norsgr}]
We show that for any given $\pi'\in\Pi_{\disc}(G'(\A))$ we can find $\pi\in\Pi_{\disc}(G(\A))$ such that $\pi'$ is a subrepresentation
of $\pi\rest_{G'(\A)}$, $\level(\pi)\ll\level(\pi')$ and $\param(\pi_\infty)\ll\param(\pi'_\infty)$.
By \cite[Theorem 4.13, Remark 4.23]{MR2918491} there exists $\pi\in\Pi_{\disc}(G(\A))$ such that $\pi'$ occurs in $\pi\rest_{G'(\A)}$.
Let $T=G/G'$. As in the proof of \cite[Lemma 5.1]{MR3352530} we have
$\level_v(\pi'_v) \ge \min_{\omega_v}\level_v(\pi_v\otimes\omega_v)$, where $\omega_v$ ranges over the characters
of $G(F_v)/G'(F_v)$, and similarly for $\param(\pi'_\infty)$.
Since the map $G(F_v)/G'(F_v)\rightarrow T(F_v)$ is injective, we can vary over the characters of $T(F_v)$ instead.
Hence the claim follows from Lemma \ref{lem: torus}.

Finally, since $\pi'$ occurs in $\pi\rest_{G'(A)}$ we have $S(\pi')\subset S(\pi)$ and $L^{S(\pi)}(s,\pi,r)=L^{S(\pi)}(s,\pi',r')$.
It is now easy to conclude the assertion of the lemma from Remark \ref{rem: altproc}.
\end{proof}

\section{Global normalizing factors and $L$-functions} \label{TWN}
In this section we consider the global normalizing factors associated to intertwining operators. We switch therefore to a slightly different setting.
Let $G$ be an isotropic reductive group defined over $F$, and $M$ a proper Levi subgroup of $G$ containing a fixed maximal $F$-split torus $T_0$.
As usual, we let $\PPP(M)$ be the set of all parabolic subgroups of $G$, defined over $F$, with Levi subgroup $M$.
For any $P\in\PPP(M)$ with unipotent radical $N_P$ let $\AF^2(P)$ be the space of automorphic forms $\varphi$ on $N_P(\A)M(F)\bs G(\A)$
such that $\modulus_P^{-\frac12}\varphi(\cdot k)\in L^2_{\disc}(A_M M(F)\bs M(\A))$ for all $k\in\K$, where
$\modulus_P$ denotes the modulus function of $P(\A)$.
For each pair $P,Q\in\PPP(M)$ there is a global intertwining operator
\[
M_{Q|P} (\lambda): \AF^2(P) \to \AF^2(Q),
\]
which has meromorphic continuation in $\lambda\in\aaa_{M,\C}^*$ \cite[\S 1]{MR681738}.

Let $\rts_M \subset \aaa_M^*$ be the set of reduced roots of $T_M$ on the Lie algebra of $G$,
and for any $P \in\PPP(M)$ let $\rts_P \subset \rts_M$ be the set of reduced roots of $T_M$ on the Lie algebra $\nnn_P$ of $N_P$.
We say that two parabolic subgroups $P,Q\in\PPP(M)$ are adjacent along $\alpha \in \rts_M$,
and write $P|^\alpha Q$, if $\rts_P\cap-\rts_Q=\{\alpha\}$. The study of the operators $M_{Q|P} (\lambda)$ reduces to this case.
Let $\pi\in\Pi_{\disc}(M(\A))$.
We are interested in the corresponding normalizing factors $n_\alpha(\pi,s)$ introduced by Langlands and Arthur
(see \cite[\S6]{MR681738}, \cite{MR999488}). They are meromorphic functions of a complex variable $s$,
and are closely connected to certain automorphic $L$-functions studied by Langlands and Shahidi.
(As the notation suggests, $n_\alpha$ depends only on $\alpha$ and not on the choice of parabolic subgroups $P|^\alpha Q$.
The precise definition of $n_\alpha$ and its relation to the intertwining operators $M_{Q|P} (\lambda)$ is described in [ibid.].
%, and summarized in \cite[\S 4]{MR3352530}.
We will not need it here.)

To describe the relevant representations of the $L$-group $\LM$,
let $U_\alpha$ be the unipotent subgroup of $G$ corresponding to $\alpha$
(so that the eigenvalues of $T_M$ on the Lie algebra of $U_\alpha$ are positive integer multiples of $\alpha$).
Let $M_\alpha$ be the group generated by $M$ and $U_{\pm\alpha}$. It is a Levi subgroup of $G$ defined over $F$ containing
$M$ as a co-rank one Levi subgroup.
Let $\prpr{M_\alpha}$ be the subgroup of $M_\alpha\cap G^{\der}$ generated by $U_{\pm\alpha}$.
By \cite[Proposition 4.11]{MR0207712} $\prpr{M_\alpha}$ is a connected normal subgroup of $M_\alpha$ defined over $F$.
Hence $\modM:=M\cap\prpr{M_\alpha}$ is a normal subgroup of $M$.
Moreover, since $M$ has co-rank one in $M_\alpha$, precisely
one simple root $\beta$ of $M_\alpha$ restricts to $\alpha$, which implies that the root system of $\prpr{M_\alpha}$ is the irreducible
component of the root system of $M_\alpha$ containing $\beta$.
The group $\prpr{M_\alpha}$ is therefore $F$-simple, and $\modM$ is a Levi subgroup of co-rank one.
(In particular, $\modM$ is connected.)
Let $\prpr{M_\alpha}^{\SC}$ be the simply connected cover of $\prpr{M_\alpha}$,
and $\widetilde\scprj:\prpr{M_\alpha}^{\SC}\rightarrow \prpr{M_\alpha}$ the natural projection.
By abuse of notation we write $\modM^{\SC}=(\scprj)^{-1}(\modM)$. (Of course $\modM^{\SC}$ is not semisimple.)
We denote by $\scprj:\modM^{\SC}\rightarrow\modM$ the restriction of $\widetilde\scprj$ to $\modM^{\SC}$.
We note that $\scprj(\modM^{\SC}(\A))=\widetilde\scprj(\prpr{M_\alpha}^{\SC}(\A))\cap\modM(\A)$.

The $L$-group $\LM_\alpha$ of $M_\alpha$ admits a parabolic subgroup $\LP_\alpha= \LM\ltimes\LU_\alpha$ corresponding to $\alpha$ whose Levi part $\LM$
is the $L$-group of $M$. The representation of $\LM$ relevant for the theory of intertwining operators is
(the contragredient of) the
adjoint representation of $\LM$ on $\Lie(\LU_\alpha)$.
We have a sequence of homomorphisms of reductive groups
\[
\prpr{M_\alpha}^{\SC}\relbar\joinrel\twoheadrightarrow \prpr{M_\alpha}\lhook\joinrel\longrightarrow M_\alpha,
\]
which restricts to
\[
\modM^{\SC}\relbar\joinrel\twoheadrightarrow\modM\lhook\joinrel\longrightarrow M.
\]
The dual sequences of homomorphisms of $L$-groups are
\[
\LM_\alpha \longrightarrow {^L\prpr{M_\alpha}} \longrightarrow
{^L(\prpr{M_\alpha}^{\SC})} = ({^L\prpr{M_\alpha}})^{\operatorname{ad}}
\]
and
\[
\LM \longrightarrow {^L\modM} \longrightarrow  {^L}(\modM^{\SC}) = {^L\modM} / Z ({^L\prpr{M_\alpha}}),
\]
respectively.
%The inclusion $\modM \hookrightarrow M$ induces a corresponding homomorphism of $L$-groups $^LM \to {^L\modM}$, and
The adjoint action of $\LM$ on $\Lie(\LU_\alpha)$ clearly factors through the composed homomorphism.
We decompose the contragredient of the adjoint representation of $^L\modM$ on $\Lie(\LU_\alpha)$ as
$\oplus_{j=1}^{\unilen}r_j$, where the (irreducible) subrepresentation $r_j$ is the eigenspace of weight $-j$ of the fundamental
weight corresponding to $\alpha$, considered as a cocharacter of the center of $^L\modM / Z ({^L\prpr{M_\alpha}})$.

The normalizing factor $n_\alpha (\pi, s)$ is closely related to the $L$-functions $L^S (j s,\pi,r_j)$, $j=1,\dots,\unilen$, which
emerged in the famous computation by Langlands of the constant term of the corresponding Eisenstein series \cite{MR0419366}.
Langlands used this to show the meromorphic continuation of these $L$-functions to the entire complex plane
(at least in the cuspidal case, but the general case follows from the cuspidal case as in the argument of Lemma \ref{lem: cuspenuf}).
These $L$-functions are also known to have finite order as meromorphic functions.
In the cuspidal case this is \cite[Theorem 2]{MR2230919}, which is based on the results of M\"uller \cite{MR1025165, MR1800065}.
The general case follows once again from the argument of Lemma \ref{lem: cuspenuf}.
However, the finer analytic properties of $L^S (s,\pi,r_j)$,
such as properties (FE+) and (CC) considered in this paper, are more elusive (cf. Remark \ref{rem: Shahidi} below).

We now summarize the pertinent properties of the normalizing factors $n_\alpha(\pi,s)$. The first one is the functional equation
\[
n_\alpha(\pi,s)\overline{n_\alpha(\pi,-\bar s)}=1,
\]
which is equivalent to
\[
\abs{n_\alpha(\pi,\iii t)}=1, \quad \quad t\in\R.
\]
The second is a factorization
\[
n_\alpha(\pi,s)=\prod_v n_{\alpha,v}(\pi,s)
\]
as an absolutely convergent product for $\Re s$ sufficiently large. The local factors $n_{\alpha,v} (\pi,s)$ are assumed to satisfy the following properties.
\begin{enumerate}
\item For all finite $v$, $n_{\alpha,v} (\pi,s)$ is a rational function in $X=q_v^{-s}$, whose degree is bounded in terms
of $G$ only and which is regular and non-zero at $X=0$.
\item If $v\in S_\infty$, then
\[
n_{\alpha,v} (\pi,s)= c_v \prod_{i=1}^{N_v} \frac{\Gamma_\R (j_i s + \alpha_i)}{\Gamma_\R (j_i s + \alpha_i + 1)},
\]
where $c_v \neq 0$, $\alpha_1,\dots,\alpha_{N_v}\in\C$
and the integers $N_v\ge1$ and $j_i\ge1$, $i = 1, \ldots, N_v$, are bounded in terms of $G$ only.
\item \label{item: pullv}
Write $\pi=\otimes\pi_v$.
If $v$ is finite, $\modM$ is unramified at $v$ and the pull-back $\pi_v \circ \scprj$ of $\pi_v$ to $\modM^{\SC}(F_v)$ %\subset \prpr{M_\alpha}^{\SC} (F_v)$
contains an unramified subrepresentation $\sigma_v$, then
\begin{equation} \label{Equationnalphav}
n_{\alpha,v} (\pi,s)=\prod_{j=1}^{\unilen}\frac{L_v (js,\sigma_v,r_j)}{L_v (js+1,\sigma_v,r_j)}.
\end{equation}
\end{enumerate}

\begin{remark}
The first and second property are clear from \cite{MR999488} (and moreover $n_{\alpha,v}(\pi,s)$ depends only on $\pi_v$, or in fact on $\pi_v\circ\scprj$
-- cf. \cite[Theorem 3.3.4]{MR3163476}).
For $v\in S_\infty$, by \cite[\S 3]{MR999488} we have in fact
\[
n_{\alpha,v} (\pi_v,s)=\prod_{j=1}^{\unilen}\frac{L_v(js,r_j\circ\phi_v)}{L_v(js+1,r_j\circ\phi_v)},
\]
where $\phi_v:W_v\rightarrow\LM$ is the Langlands parameter associated to $\pi_v$. However, we will only need the qualitative property stated above.
The third property is implicit in \cite{MR999488}. We omit the details.
\end{remark}

\begin{remark} \label{RemarkLocalnalpha}
%Write $\prpr{M_\alpha}(F_v)^+ \subset \prpr{M_\alpha} (F_v)$ for the image of $\prpr{M_\alpha}^{\SC}(F_v)$ under $\scprj$,
%and similarly for $\prpr{M_\alpha} (\A)^+ \subset \prpr{M_\alpha} (\A)$. We note that the group
%$\prpr{M_\alpha} (\A)^+$ is the closed subgroup of $\prpr{M_\alpha} (\A)$ generated by $U_{\pm\alpha} (\A)$ \cite[Proposition 6.2]{MR0316587}.
%It is also the derived group of $\prpr{M_\alpha} (\A)$ \cite[Theorem 7.1]{MR1278263}.

We can rewrite the assumption on $\pi_v$ in \eqref{item: pullv} in an equivalent way by saying that
$\pi_v|_{\modM (F_v)}$ contains an irreducible subrepresentation of the form
$\sigma_v \chi_v$, where $\sigma_v$ is an unramified representation of $\modM(F_v)$ and $\chi_v$ is
a character of $\modM(F_v)/\scprj(\modM^{\SC}(F_v))$.
%the restriction to $\modM(F_v)$ of a character of the finite abelian group $\prpr{M_\alpha} (F_v) / \prpr{M_\alpha} (F_v)^+$.
%
%Note that the orders of the abelian groups $\modM(F_v)/\scprj(\modM^{\SC}(F_v))$
%are bounded in terms of $G$ only, as $v$ ranges over the places of $F$.
\end{remark}

We now introduce the property (TWN+), which is the main object of this paper.

\begin{definition} \label{DefinitionTWN}
The group $G$ satisfies property (TWN+) (tempered winding numbers, strong version) if, for any proper Levi subgroup $M$ of $G$ defined over $F$,
and any root $\alpha \in \rts_M$ we have the estimate
\[
\int_T^{T+1}\abs{n_\alpha'(\pi,\iii t)}\ dt\ll\log(\abs{T}+\param(\pi_\infty;\scprj)+\level(\pi;\scprj))
\]
for all $\pi\in\Pi_{\disc}(M(\A))$ and all real numbers $T$.
\end{definition}

Recall that the invariants $\param(\pi_\infty;\scprj)$ and $\level(\pi;\scprj)$ are defined in \S\ref{sec: relparams}.
(Of course, the ambient group is $M$ in our case.)
As explained in \cite[Remark 5.3, part 2]{MR3352530}, this property implies property (TWN) for $G$ formulated in [ibid., Definition 5.2],
which is relevant to the limit multiplicity problem.

In view of the description of the unramified factors of $n_\alpha (\pi,s)$, it is no wonder that the property (TWN+)
is intimately related to analytic properties of the automorphic $L$-functions $L^S(s,\pi,r_j)$.
In the special cases $G = \GL (n)$ and $G = \SL (n)$ this was spelled out in [ibid., Proposition 5.5].
Here, we will analyze it in the general case.

\begin{definition}
We say that $G$ satisfies property (L), if for any standard Levi subgroup $M$, any $\alpha \in \rts_M$, and any irreducible
constituent $r=r_j$ of the representation $(\Ad_{\LU_\alpha})^\vee$ of the group $^L\modM$, the pair $(\modM,r)$
satisfies properties (FE+) and \CC.
\end{definition}

\begin{remark} \label{rem: onlyderived}
It is clear from the definition that property (L) depends only on the derived group of $G$.
\end{remark}

\begin{remark} \label{rem: r1}
We expect that every group satisfies property (L).
Using \cite[Proposition 3.1]{MR1800349}, it is possible to reduce this question to the constituent $r_1$.
\end{remark}

\begin{remark} \label{rem: Shahidi}
Suppose that $G$ is quasi-split and that $\pi$ is a \emph{generic} cuspidal
representation of $M (\A)$. (If $M$ is isogenous to a product of general linear groups, then this condition is satisfied for all cuspidal representations
$\pi$.) By the results of Shahidi, the meromorphic function $L^{S(\pi)}(s,\pi,r)$ satisfies then a functional equation of the form \eqref{eq: FE}.
Moreover, local factors $\gamma_v^{\Sh}(s,\pi_v,r)$ are defined and satisfy \eqref{eq: global FE}, \GF\ and \eqref{eq: comparch} \cite{MR1070599}.
Shahidi's work also gives that under some mild assumptions which are known in almost all cases (see \cite{MR2767519} and the references therein),
these $L$-functions admit finitely many poles and are of order one (see \cite{MR1800349}, with some complements in \cite{MR2230919}).
However, the poles are controlled by those of the corresponding Eisenstein series in the right-half plane,
and in general it is not clear how to bound the number of the latter in terms of $G$ only.
(This is known in several cases and is expected to hold in general.)
Therefore, even granted the reduction to the generic case, without additional input we cannot conclude from Shahidi's work by itself that $r$ has property (FE+).

In addition, it is not clear how to approach property \eqref{eq: locndbnd} using the Langlands-Shahidi method (although it is certainly not excluded
that this is possible).
%Although the stability properties of the factors
%$\gamma_v^{\Sh}(s,\pi_v,r)$ have been studied in some important cases
%(cf. \cite{MR2398146}, \cite{	1412.1448}), it is also unclear whether holds for them in general.

In the cases at hand we will supplement the information from Shahidi's work by using integral representations of Rankin--Selberg type,
which give better control of the poles of (at least) the partial $L$-functions as well as of the local $\gamma$-factors.
\end{remark}

We now have the following implication.

\begin{proposition} \label{prop: ABCimpliesTWN}
Suppose that $G$ satisfies property (L). Then it satisfies property (TWN+).
\end{proposition}

This result essentially follows from Proposition \ref{lem: logderbnd}.
We first need a simple lemma to account for the ramified local factors.

\begin{lemma} \label{lem: estimfctrs}
Let $P$ be a polynomial of degree $d$ and let $f(s)=P(p^{-s})$. Then
\[
\int_T^{T+1}\abs{\Re\frac{f'(\iii t)}{f(\iii t)}}\ dt\le d(\pi+\log p)
\]
for all $T\in\R$.
Similarly let $f(s)=\frac{\Gamma_{\R}(s+\alpha)}{\Gamma_{\R}(s+1+\alpha)}$ for some $\alpha\in\C$. Then
\[
\int_T^{T+1}\abs{\Re\frac{f'(\iii t)}{f(\iii t)}}\ dt\ll1
\]
for all $T\in\R$.
\end{lemma}

\begin{proof}
For the first part, it is enough to consider the case $P(z)=1-\alpha z$ with $\alpha\in\C$.
(The left-hand side vanishes for $P(z)=z$.)
Absorbing the argument of $\alpha$ into $T$, we may also assume without loss of generality that $\alpha\in\R$.
Note that then
\begin{eqnarray*}
\log p\int\abs{\Re\frac{1}{1-\alpha p^{-\iii t}}-\frac12}\ dt & = &
\frac12\log p\int \frac{\abs{1-\alpha^2}}{1+\alpha^2-2\alpha\cos(t\log p)}\ dt\\ & = &
\arctan(\abs{\frac{1+\alpha}{1-\alpha}}\tan(\frac{t\log p}2))+C,
\end{eqnarray*}
and hence, since the integrand is periodic with period $2\pi/\log p$,
\[
\log p\int_T^{T+1}\abs{\Re\frac1{1-\alpha p^{-\iii t}}-\frac12}\ dt\le \pi+\frac12\log p.
\]
Thus,
\[
\int_T^{T+1}\abs{\Re\frac{f'(\iii t)}{f(\iii t)}}\ dt=
\log p\int_T^{T+1}\abs{\Re\frac{\alpha p^{-\iii t}}{1-\alpha p^{-\iii t}}}\ dt\le\pi+\log p.
\]

For the second part we can assume once again that $\alpha\in\R$. Note that
\[
\frac{\Gamma_{\R}'(z)}{\Gamma_{\R}(z)}-\frac{\Gamma_{\R}'(z+1)}{\Gamma_{\R}(z+1)}=
\sum_{k=0}^\infty\big(\frac1{z+2k+1}-\frac1{z+2k}\big)=
\sum_{k=0}^\infty\frac1{(z+2k)(z+2k+1)}.
\]
Thus,
\[
\frac{\Gamma_{\R}'(\alpha+\iii t)}{\Gamma_{\R}(\alpha+\iii t)}-\frac{\Gamma_{\R}'(\alpha+\iii t+1)}{\Gamma_{\R}(\alpha+\iii t+1)}=
A+B,
\]
where upon writing $\alpha=n+\delta$ with $n\in\Z$ and $-\frac12<\delta\le\frac12$, we have
\[
A=\begin{cases}\frac{(-1)^{n+1}}{\delta+\iii t}&\text{ if }\alpha<\frac12,\\
0&\text{otherwise,}\end{cases}
\]
and
\begin{eqnarray*}
\abs{B} & \le & \frac1{\abs{\delta+\iii t+(-1)^n}}+\sum_{k \ge 0, \, k\ne\lfloor\frac{-n}2\rfloor}\frac1{\abs{\alpha+\iii t+2k}\abs{\alpha+\iii t+2k+1}}\\ & \le &
\frac1{\abs{\delta+(-1)^n}}+\sum_{k\ne\lfloor\frac{-n}2\rfloor}\frac1{\abs{\alpha+2k}\abs{\alpha+2k+1}}\ll1.
\end{eqnarray*}
It remains to note that
\[
\int_T^{T+1}\abs{\Re\frac1{\delta+\iii t}}\ dt=\abs{\delta}\int_T^{T+1}\frac{dt}{\delta^2+t^2}<\pi. \qedhere
\]
\end{proof}

\begin{proof}[Proof of Proposition \ref{prop: ABCimpliesTWN}]
Assume that $G$ satisfies property (L), and let $M$, $\alpha$, and $\pi\in\Pi_{\disc}(M(\A))$ be given.
We apply Lemma \ref{lem: restH} with respect to $\scprj:\modM^{\SC}\rightarrow\modM\triangleleft M$.
Let $\sigma\in\Pi_{\disc}(\modM(\A))$ and $\chi$ be as in that lemma.
Then by Remark \ref{RemarkLocalnalpha} we have
\begin{equation} \label{Comparenalpha}
\prod_{v \notin S (\sigma)} n_{\alpha,v} (\pi, s) =\prod_{j=1}^{\unilen}\frac{L^{S (\sigma)}(js,\sigma,r_j)}{L^{S (\sigma)}(js+1,\sigma,r_j)}.
\end{equation}
As in \S\ref{sec: redLfunction} define
\begin{equation} \label{Comparemsigma}
m(\sigma,s)=\prod_{j=1}^{\unilen}\frac{\epsilon^{\red}(1,\sigma,r_j)L^{\red}(js,\sigma,r_j)}{L^{\red}(js+1,\sigma,r_j)}.
\end{equation}
By Proposition \ref{lem: logderbnd} we have $\abs{m(\sigma,\iii t)}=1$ for all $t\in\R$, and
\[
\int_T^{T+1}\abs{m'(\sigma,\iii t)}\ dt\ll \log (\abs{T}+2) +
\sum_j \left[ \log\arithcond(\sigma,r_j)+ \log \infcond_\infty(\sigma,r_j) \right]
+ \log \level (\sigma).
\]
By condition (CC) for $(\modM, r_j)$ and Lemma \ref{lem: restH}, for all $j$ we have here
\[
\log \arithcond(\sigma,r_j) \ll \log \level (\sigma) + 1\ll \log\level(\pi;\scprj) + 1,
\]
while
\[
\log \infcond_\infty(\sigma,r_j) \ll \log\param^{\modM}(\sigma_\infty) + 1=\log\param(\pi_\infty;\scprj) + 1.
\]

It remains to compare $m(\sigma,s)$ and $n_\alpha (\pi,s)$.
Consider the quotient $\phi(s)=m(\sigma,s)/n_\alpha (\pi,s)$. We have $\abs{\phi(\iii t)}=1$, and therefore
$\phi'(\iii t)/\phi(\iii t)\in\R$ for $t\in\R$.
On the other hand, from \eqref{Comparenalpha} and \eqref{Comparemsigma} we get that
\[
\phi(s)= c \prod_{j=1}^{\unilen}\prod_{p\in S_{\Q} (\sigma)}\frac{L_p^{\red}(js,\sigma,r_j)}{L_p^{\red}(js+1,\sigma,r_j)}
\prod_{v\in S(\sigma)} n_{\alpha,v}(\pi,s)^{-1}
\]
for some non-zero constant $c$.
By the nature of the local factors $L_p^{\red}(js,\sigma,r_j)$ and $n_{\alpha,v}(\pi,s)$, we may conclude from Lemma \ref{lem: estimfctrs}
and \eqref{eq: levelSpi} that
\begin{eqnarray*}
\int_T^{T+1}\abs{\phi'(\iii t)}\ dt=\int_T^{T+1}\abs{\Re\frac{\phi'(\iii t)}{\phi(\iii t)}}\ dt & \ll_G & [F:\Q]+\sum_{v\in S_f(\sigma)}\log q_v \\
& \ll & \log\level(\pi;\scprj) + 1.
\end{eqnarray*}
This concludes the proof of the proposition.
\end{proof}

\begin{remark} \label{rem: wABC}
For Proposition \ref{prop: ABCimpliesTWN} to hold, we may replace the condition (FE+) in the definition of property (L)
by the weaker condition that (FE+) holds virtually (cf. Remark \ref{rem: weaknice is enough}).
It would be interesting to know whether one can further weaken the assumptions in Proposition \ref{prop: ABCimpliesTWN}.
\end{remark}

%We can summarize the discussion of \S\ref{sec: goodLexamples} as follows:
In the next sections we will prove:

\begin{theorem} \label{thm: ABC}
The following groups satisfy property (L), and hence also property (TWN+).
\begin{enumerate}
\item $\GL (n)$ and its inner forms.
%\item $\SL (n)$ and its inner forms.
\item Quasi-split classical groups.
\item The exceptional group $G_2$.
\end{enumerate}
Thus, by Remark \ref{rem: onlyderived}, the same holds for any group whose derived group coincides with the derived group of any one of the groups above.
\end{theorem}

The proof is based on a case-by-case analysis of the $L$-functions appearing in the definition of property (L).
The quasi-splitness assumption in part 3 comes from the fact that we use functoriality to $\GL (n)$, which at the moment is only available in this case.
%It will occupy the remaining part of the paper.
(See Remark \ref{rem: extend to non-qs} below.)
In each case we will use Corollary \ref{cor: enufsomegamma} for appropriately defined local factors.
We will also often use Lemma \ref{lem: cuspenuf} to reduce to the cuspidal case.

\section{Inner forms of $\GL(n)$} \label{sec: goodLexamples}
%In this section we examine certain $L$-functions whose analytic properties are known and show that they satisfy properties (FE+) and \CC.

We first consider the groups $\GL(n)$ and their inner forms.
In order to prove Theorem \ref{thm: ABC} in this case, in view of Lemma \ref{lem: norsgr} it suffices to show the following

\begin{theorem} \label{thm: goodLfunctinos}
Let $G=G_1\times G_2$, where $G_i$ is an inner form over $F$ of $\GL ({n_i})$, $i=1,2$, and $r=\Std_{n_1}\otimes\Std_{n_2}$
where $\Std_n$ is the standard $n$-dimensional representation of $\GL(n,\C)$.
Then the pair $(G,r)$ satisfies properties (FE+) and \CC.
%The following pairs $(G,r)$ satisfy properties (FE+) and \CC:
%\begin{enumerate}
%\item $G=\GL_n$, $r$ is the standard $n$-dimensional representation $\Std_n$.
%\item $G$ is an inner form of $\GL (n)$ and $r=\Std_n$.
%\item $G=\GL ({n_1}) \times \GL ({n_2})$, $r$ is the tensor product representation $\Std_{n_1}\otimes\Std_{n_2}$.
%\item $G=G_1\times G_2$, where $G_i$ is an inner form of $\GL ({n_i})$, $i=1,2$, and $r=\Std_{n_1}\otimes\Std_{n_2}$.
%\end{enumerate}
\end{theorem}

This theorem will be proved in the rest of this section.

\subsection{}
We start with the case $n_2=1$. Let $G$ be an inner form of $\GL(n)$ and let $\pi\in\Pi_{\cusp}(\GL(n,\A))$.
For any Hecke character $\chi$ of $F$ we have
\[
L^S(s,\pi\times\chi,\Std_n\times\Std_1)=L^S(s,\pi\otimes\chi,\Std),
\]
where $\pi\otimes\chi$ denotes the twist of $\pi$ by $\chi$. Therefore, it suffices to show that $(G,\Std_n)$ satisfies (FE+) and \CC.
The standard $L$-function for $\pi$ was studied by Godement--Jacquet \cite{MR0342495}.
%\Erez{The reference \cite{MR546609} is only for $D=F$.}
%\Erez{General reduction from discrete series to cuspidal}
In particular, $\Std_n$ has property (FE+) and the local factors
\[
\gamma^{\GJ}(s,\pi_v,\Std_n)=\epsilon^{\GJ}(s,\pi_v,\Std_n)\frac{L^{\GJ}(1-s,\pi_v^\vee,\Std_n)}{L^{\GJ}(s,\pi_v,\Std_n)}
\]
defined by Godement--Jacquet (with respect to the fixed character $\psi_v$ as in \S\ref{sec: psivconv}) satisfy \eqref{eq: global FE}, \GF\ and
\eqref{eq: comparch} (and in particular \eqref{eq: archbnd}).
(Of course, in the case $n=1$ these are Tate's local factors.)
%In fact, $L(s,\pi_v,\Std_n)$ and $\epsilon(s,\pi_v,\Std_n)$ are defined.
For brevity we write for $v$ finite
\[
\cond(\pi_v):=\expcond_v^{\GJ}(\pi_v,\Std_n)
\]
(see \eqref{eq: gammadot}).
This is the usual conductor of $\pi_v$.

%Note that this differs slightly from the usual terminology (e.g., \Erez{reference}) where
%$\cond(\pi_v)$ refers to the exponent of standard epsilon factor (rather than the gamma factor) with respect to a character $\psi_v$
%whose conductor is $\OO_v$ (rather than the dual of $\OO_v$ with respect to $\Tr_{F_v/\Q_p}$). The difference however is bounded \Erez{in terms of..}.
%\Erez{Not consistent with standard notation because of choice of $\psi_v$; also there is a difference between epsilon factor
%and gamma factor}
%The relation
%\begin{equation} \label{eq: condlelvl}
%\cond(\pi_v)\le\log_{q_v}\level_v(\pi_v)
%\end{equation}
%follows from \cite{MR620708} -- see \cite[Lemma 5.6]{FLM3}. \Erez{relevant reference?}
%Hence \CC\ is satisfied.
% \eqref{eq: locndbnd} (with $\le$ instead of $\ll$)
%Moreover, the conductor is bounded by the level \cite{MR620708}.

Property \eqref{eq: locndbnd} for $(G,\Std_n)$ follows from the following standard result.\footnote{In the case where $D=F$
and $\pi_v$ is generic, it follows from the well-known description of $\cond(\pi_v)$
due to Jacquet--Piatetski-Shapiro--Shalika \cite{MR620708}, that $\cond(\pi_v) \ge \log_{q} \level_v (\pi_v)$.
In \cite[Lemma 5.6]{MR3352530} it was asserted that we should always have
$\cond(\pi_v) \le \log_{q} \level_v (\pi_v)$ for $D=F$, which is not true. After replacing this incorrect conductor bound by the upper bound $n \log_{q} \level_v (\pi_v)$ of the current Lemma \ref{lem: uprcondva}, the rest of the argument of [ibid.] goes through without any change.}

\begin{lemma} \label{lem: uprcondva}
Let $F$ be a non-archimedean field with residue field $\F_q$, $D$ a central division algebra over $F$ of degree $d$
and $G=\GL (m,D)$.
Let $\OO_D$ be the maximal order of $D$ and let $\varpi_D$ be a prime element of $D$.
Let $\pi$ be an irreducible representation $\pi$ of $G$ and let
\[
l=\min\{i\ge0:\pi^{K_i}\ne0\},
\]
where $K_0=\GL(m,\OO_D)$ and $K_i=1+\varpi_D^iM_m(\OO_D)$, $i>0$.
Then $\cond(\pi)\le n+m(l-1)$ where $n=md$.
\end{lemma}

\begin{proof}
The result follows from the relation between conductor and depth proved recently
in \cite{1311.1606}. (See also \cite{MR772712, MR882297, MR1110390, MR1950297, MR2303528} and the appendix of \cite{MR1915174}.)
%The result certainly follows from more precise characterizations of $\cond(\pi)$. (See Remark \ref{rem: on conductor} below.)
For convenience we provide an easy, self-contained argument.

If $l=0$, i.e. if $\pi^{K_0}\ne0$, then $\cond(\pi)=m(d-1)=n-m$ \cite[\S I.6]{MR0342495}, and the lemma is clear.
Assume therefore that $l\ge1$.
Let $\Nrd$ (resp., $\Trd$) be the reduced norm (resp., trace) on $M_m(D)$.
For $\Phi\in\swrz(M_m(D))$ and a matrix coefficient $f$ of $\pi$ let
\[
Z(s,f,\Phi)=\int_Gf(g)\Phi(g)\abs{\Nrd(g)}^{s+\frac{n-1}2}\ dg.
\]
Consider the functional equation
\begin{equation} \label{eq: GJFE}
\frac{Z(1-s,f^\vee,\hat\Phi)}{L^{\GJ}(1-s,\pi^\vee)}=(-1)^{n-m}\epsilon^{\GJ}(s,\pi,\psi)\frac{Z(s,f,\Phi)}{L^{\GJ}(s,\pi)},
\end{equation}
where $\hat\Phi$ is the Fourier transform of $\Phi$ given by
\[
\hat\Phi(x)=\int_{M_m(D)}\Phi(y)\psi(\Trd(xy))\ dy
\]
for a suitable Haar measure on $M_m(D)$, a non-trivial additive character $\psi$ of $F$,
and $f^\vee(g)=f(g^{-1})$ (a matrix coefficient of $\pi^\vee$).
Take $f$ to be a bi-$K_l$-invariant matrix coefficient of $\pi$ such that $f(1)=1$
and $\Phi=1_{K_l}$. Then $Z(s,f,\Phi)\equiv\vol(K_l)$ so that $\gamma^{\GJ}(s,\pi)$ is a scalar multiple of $Z(1-s,f^\vee,\hat\Phi)$.
On the other hand, if $\psi$ has conductor $\OO_F$ then $\hat\Phi$ is supported in $\varpi_D^{1-d-l}M_m(\OO_D)$ \cite[Ch. X, \S2, Proposition 5]{MR0427267}.
Thus, $Z(1-s,f^\vee,\hat\Phi)$ is a Laurent series $\sum a_ix^i$ in $x=q^{-s}$ such that $a_i=0$ for all $i>m(d+l-1)$.
Hence, the left-hand side of \eqref{eq: GJFE} is a Laurent polynomial $\sum b_ix^i$ with $b_i=0$ for all $i>m(d+l-1)$.
Since $\epsilon^{\GJ}(s,\pi,\psi) = \epsilon^{\GJ}(0,\pi,\psi) x^{\cond(\pi)}$,
we conclude that $\cond(\pi)\le m(d+l-1)=n+m(l-1)$, as required.
\end{proof}

We also remark that for $v\in S_\infty$ we have
\begin{equation} \label{eq: archstd}
\param(\pi_v)\ll\infcond_\infty(\pi,r)^2\ll\param(\pi_v)^n.
\end{equation}

\subsection{}
Next consider $G=\GL ({n_1}) \times \GL ({n_2})$ with the tensor product representation $r=\Std_{n_1}\otimes\Std_{n_2}$.
This $L$-function was studied independently by Rankin and Selberg for $n_1=n_2=2$ (\cite{MR0001249, MR0002626}) and in the general case
by Jacquet, Piatetski-Shapiro and Shalika
(\cite{MR519356, MR528964, MR549066, MR633552, MR618323, MR623137, MR701565, MR1159102, MR2533003}, see also \cite{MR1990380, MR2508768}
and the references therein).
In particular, they showed that for $\pi$, $\sigma$ cuspidal the function
\[
s(s-1)L^{S(\pi)\cup S(\sigma)}(s,\pi\times\sigma,r)
\]
is an entire function of order one. Alternatively, this $L$-function can be also studied using the Langlands--Shahidi method.
The local factors arising from either method coincide.
We denote them by
\[
\gamma_v^{\JPSS}(s,\pi_v\times\sigma_v,r)=\epsilon_v^{\JPSS}(s,\pi_v\times\sigma_v,r)\frac{L^{\JPSS}(1-s,\pi_v\times\sigma_v,r^\vee)}
{L^{\JPSS}(s,\pi_v\times\sigma_v,r)},
\]
where again implicitly $\psi_v$ are as in \S\ref{sec: psivconv}.
These $\gamma$-factors satisfy \eqref{eq: global FE}, \GF\ and \eqref{eq: comparch}. Thus $r$ has property (FE+).
In the non-archimedean case we set for brevity $\cond(\pi_v\times\sigma_v)=\expcond_v^{\JPSS}(\pi_v\times\sigma_v,r)$
(see \eqref{eq: gammadot}).
By a result of Bushnell--Henniart \cite{MR1462836} we have
\begin{equation} \label{eq: BH}
\cond(\pi_v\times\sigma_v)\le n_2\cond(\pi_v)+n_1\cond(\sigma_v).
\end{equation}
(It is worthwhile to mention that this result does not depend on the Bushnell--Kutzko classification of supercuspidal representations
\cite{MR1204652}, unlike the lower bound for $\cond(\pi_v\times\sigma_v)$ proved in \cite{MR1606410}.)
%See also \S\ref{sec: supercuspidal support}.)
By Lemma \ref{lem: uprcondva} we conclude that property \eqref{eq: locndbnd} holds for $\gamma_v^{\JPSS}(s,\pi_v\times\sigma_v,r)$.
%(See also \cite[Lemma 5.7]{MR3352530}.)
%%%Hence \CC\ is also satisfied.

\subsection{}
To finish the proof of Theorem \ref{thm: goodLfunctinos} in the general case,
we first recall the Jacquet--Langlands correspondence \cite{MR0401654, MR771672}, proved in this generality by Badulescu--Grbac \cite{MR2390289},
using results of Arthur--Clozel \cite{MR1007299}, with some complements in \cite{MR2684298}.
Let $G'=\GL (n)$ and let $G$ be an inner form of $G'$.
Then for any $\pi\in\Pi_{\disc}(G(\A))$ there exists a unique $\pi'\in\Pi_{\disc}(G'(\A))$
such that $\pi'_v=\pi_v$ for all $v$ where $G'$ splits. The representation $\pi'$ is called the Jacquet--Langlands transfer of $\pi$
and will be denoted by $\JL(\pi)$.
For any place $v$ we have
\begin{equation} \label{eq: JLgamma}
\gamma_v^{\GJ}(s,\pi_v,\psi_v)=\gamma_v^{\GJ}(s,\JL(\pi)_v,\psi_v).
\end{equation}
Note that it is not true in general that $\cond(\pi_v)=\cond(\JL(\pi)_v)$.
This is related to the fact that $\JL(\pi)_v$ does not depend only on $\pi_v$. (It is true however that $\pi_v$
is determined by $\JL(\pi)_v$.)
For instance, if $\pi$ is the identity representation of $G$, then $\JL(\pi)$ is the identity representation of $G$,
so that $\cond(\JL(\pi)_v)=0$ for all finite $v$. On the other hand, if $G$ does not split at $v$, then $\cond(\pi_v)>0$.
However, as was pointed out after \eqref{eq: locndbnd}, this is immaterial for our purposes, since \eqref{eq: JLgamma} implies that in any case
\begin{equation} \label{eq: conddifference}
\abs{\cond(\pi_v) - \cond(\JL(\pi)_v)} \le n.
\end{equation}

%the property \eqref{eq: locndbnd}
%depends only on $\gamma^{\bullet}$.

%be the image of $\pi$ under the Jacquet--Langlands correspondence.
%In particular, $\omega_{\pi'}=\omega_\pi$, $S(\pi)=S(\pi')$ and $L^{S(\pi)}(s,\pi,\Std_n)=L^{S(\pi')}(s,\pi',\Std_n)$.

Now let $G=G_1\times G_2$ where $G_i$ is an inner form of $G'_i=\GL ({n_i})$, $i=1,2$. \label{sec: RSconvda}
%\[
%\gamma_v^{\GJ}(s,\pi_v\times\chi_v,\psi_v)=\gamma_v^{\GJ}(s,\pi'_v\times\chi_v,\psi_v)
%\]
%for any place $v$ and any character $\chi_v$ of $F_v^*$. \Erez{Reference}
%\Erez{Technically we didn't say anything about the twisted gamma factor}
%Looking at the proof of \cite[Proposition 2.2]{MR787183} about the stability of $\gamma$ factor
%(which as remarked in [ibid, (2.8)] applies equally well to $G$) we see that for all $v\in S(\pi)$ and for any
%character $\chi_v$ of $F_v^*$ such that $\cond(\chi_v)\ge\log_q\level(\pi_v)$ we have
%\[
%\gamma_v^{\GJ}(s,\pi_v\times\chi_v,\psi_v)=\gamma_v^{\Tate}(s,\omega_{\pi_v}\chi_v,\psi_v)\gamma_v^{\Tate}(s,\chi_v,\psi_v)^{n-1}.
%\]
%\Erez{Complete} \Erez{Define $\gamma_v^{\Tate}(s,\chi_v,\psi_v)$}
%Taking $\chi_v$ such that $\cond(\chi_v)=\log_q\level(\pi_v)$ we infer that
%\[
%\cond(\pi_v\times\chi_v)\ll\log_q\level(\pi_v).
%\]
%On the other hand, \eqref{eq: JL preserves gamma factors} implies that $\cond(\pi_v)=\cond(\pi'_v)$ and $\cond(\pi_v\times\chi_v)=\cond(\pi'_v\times\chi_v)$
%while by \cite{MR1462836} we have
%\[
%\cond(\pi'_v)\le\cond(\pi'_v\times\chi_v)+n\cond(\chi_v).
%\]
%We infer that
%\begin{equation} \label{eq: condbndbylvl}
%\cond(\pi_v)=\cond(\pi'_v)\ll\log_q\level(\pi_v)
%\end{equation}
%as required.
Let $\pi\in\Pi_{\disc}(G_1)$, $\sigma\in\Pi_{\disc}(G_2)$ and let $\JL(\pi)$ and $\JL(\sigma)$ be their
Jacquet--Langlands transfers to $G'_1$ and $G'_2$, respectively.
We have
\begin{equation} \label{eq: unrJL}
L^{S(\pi)\cup S(\sigma)}(\pi\times\sigma,\Std_{n_1}\times\Std_{n_2})=
L^{S(\pi')\cup S(\sigma')}(\JL(\pi)\times\JL(\sigma),\Std_{n_1}\times\Std_{n_2}).
\end{equation}
%%Hence $(G_1\times G_2,\Std_{n_1}\times\Std_{n_2})$ satisfies property (FE+).
Since we do not have at our disposal an independent theory of Rankin--Selberg convolutions for $G_1\times G_2$, we will resort
to the one on $G'_1\times G'_2$. That is, we simply define
\[
\gamma_v^{\JPSS\circ\JL}(s,\pi\times\sigma,\Std_{n_1}\times\Std_{n_2})=
\gamma_v^{\JPSS}(s,\JL(\pi)_v\times\JL(\sigma)_v,\Std_{n_1}\times\Std_{n_2}).
\]
(In fact, these $\gamma$-factors depend only on $\pi_v$ and $\sigma_v$, but we do not need to use this fact.)
These factors satisfy properties \eqref{eq: global FE} and \GF\ by \eqref{eq: unrJL} and the fact that $\gamma_v^{\JPSS}$ satisfy
\eqref{eq: global FE} and \GF\ for $G_1'\times G_2'$.
Property \eqref{eq: archbnd} for $\gamma_v^{\JPSS\circ\JL}$ follows from the corresponding result for $\gamma_v^{\JPSS}$
and the fact that (under a suitable identification) $\JL(\pi)_v$ has the same infinitesimal character as $\pi_v$ \cite{MR2684298}.
Combining \eqref{eq: BH}, \eqref{eq: conddifference} and Lemma \ref{lem: uprcondva} (for both $\pi_v$ and $\sigma_v$),
we conclude property \eqref{eq: locndbnd} for the factors $\gamma_v^{\JPSS\circ\JL}$.
Thus, we can apply Corollary \ref{cor: enufsomegamma} to conclude the proof of Theorem \ref{thm: goodLfunctinos}.
%%Hence \CC\ holds.
%The local Jacquet--Langlands correspondence preserve the level \cite{MR2303528}.
%The difference between conductor and level is also explained there at least for square-integrable representations.

\section{Classical groups} \label{SectionClassical}

We now consider the case of quasi-split classical groups, i.e. the second part of Theorem \ref{thm: ABC}.

By a classical group we mean either a symplectic group (which is automatically split), a special orthogonal group,
or a unitary group. In the latter case we denote by $E$ the quadratic extension of $F$ pertaining to the hermitian
form defining the unitary group (i.e., the quadratic extension over which the group splits). In all other cases let $E=F$.
The $L$-group of a classical group is equipped with a natural embedding
\[
\can_G:\LG\rightarrow\,^L{\Res_{E/F}\GL(m)},
\]
where, if $r$ is the $\bar F$-rank of $G$, then
\[
m=\begin{cases}2r+1,&\text{in the symplectic case},\\
2r,&\text{in the orthogonal case},\\
r,&\text{in the unitary case}.\end{cases}
\]
Here, in the unitary case we have
\[
^L{\Res_{E/F}\GL(m)}\simeq(\GL(m,\C)\times\GL(m,\C))\rtimes W_F,
\]
where $W_F$ acts on $\GL(m,\C)\times\GL(m,\C)$ via $\Gal(E/F)$ by permuting the factors.

%Correspondingly, we have a representation
%\[
%\can_G=p\circ\canhom_G:\LG\rightarrow\GL(N,\C)
%\]
%of dimension $N=[E:F]m$,
Let $p_n:{^L}\Res_{E/F}\GL(n)\rightarrow\GL([E:F]n,\C)$ be the representation given by the projection $^L{\GL(n)}\rightarrow\GL(n,\C)$ if $E=F$,
and, in the case $E\ne F$, by $p_n((x,y),e)=\diag(x,y)$, $x,y\in\GL(n,\C)$, and
\[
p_n(I_n,I_n,\sigma)=\begin{cases}I_{2n},&\text{if the image of $\sigma$ in $\Gal(E/F)$ is the identity,}\\
\sm{}{I_n}{I_n}{},&\text{otherwise,}\end{cases}
\]
using the description of $^L{\Res_{E/F}\GL(n)}$ above.
We also consider the representation
\[
T_{m,n}:{^L}(\Res_{E/F}\GL(m)\times\Res_{E/F}\GL(n))\rightarrow\GL([E:F]mn,\C)
\]
obtained by the composition of $p_{nm}$ with the ``tensor product'' homomorphism
\[
{^L}(\Res_{E/F}\GL(m)\times\Res_{E/F}\GL(n))\rightarrow{^L}\Res_{E/F}\GL(mn).
\]
In the case $E=F$, $T_{m,n}=\Std_m\times\Std_n$.

\begin{theorem} \label{thm: classical}
The following pairs $(G,r)$ satisfy properties (FE+) and \CC:
\begin{enumerate}
\item ($E=F$), $G=\GL (n)\times\GL(1)$, $r=\wedge^2\otimes\Std_1$, where $\wedge^2$ is the exterior square representation.
\item ($E=F$), $G=\GL (n)\times\GL(1)$, $r=\Sym^2\otimes\Std_1$, where $\Sym^2$ is the symmetric square representation.
\item ($E\ne F$), $G=\Res_{E/F}\GL (n)$, and $r=\Asai^\pm$ is either the Asai or the twisted Asai representation (cf. \cite{MR2767514}).
\item $G=G'\times\Res_{E/F}\GL (n)$, where $G'$ is a quasi-split classical group and $r=T_{m,n}\circ(\can_{G'}\otimes\id_n)$.
\end{enumerate}
\end{theorem}

This theorem implies Theorem \ref{thm: ABC} for quasi-split classical groups by the explicit description
of the representations appearing in the definition of property (L) (e.g., \cite[Appendix C]{MR2683009}).
Actually, for Theorem \ref{thm: ABC}, in the first two parts of Theorem \ref{thm: classical}
it is enough to consider the exterior square and the symmetric square representations themselves
(without the twist), but we include the slightly more general statement since it does not incur additional difficulty
and the twisted representations are relevant for the $\operatorname{GSpin}$ groups (cf.~\cite{MR3227529}). %once Arthur's work is adapted to them.
(For the Asai $L$-function there is no need to consider $\Asai^\pm\times\Std_1$, since we can incorporate the twist into the representation.)

\subsection{}
In this subsection we consider the first three cases of Theorem \ref{thm: classical}.
%consider the twisted exterior and symmetric square $L$-functions for $G=\GL (n)\times\GL(1)$, as well as the Asai $L$-function for $\Res_{E/F}\GL (n)$.
By Lemma \ref{lem: cuspenuf} and Theorem \ref{thm: goodLfunctinos} it is enough to consider the cuspidal case.
Let $\pi\in\Pi_{\cusp}(\GL(n,\A_E))$.
By Remark \ref{rem: Shahidi}, the partial $L$-function has finitely many poles and of order one,
and local factors $\gamma_v^{\Sh}(s,\pi\times\chi,r)$ are defined and satisfy \eqref{eq: global FE}, \GF\ and
\eqref{eq: comparch}.
%The meromorphic continuation and functional equation is a special case of the results of Shahidi \cite{MR1070599}.
%Shahidi also defines local factors $\gamma_v^{\Sh}(s,\pi\times\chi,r)$ and proves the property \AF\ for them.
%\Erez{Makes this a general remark?}

%We first recall that the poles of the $L$-function for $\Re s>0$ are precisely the poles of the Eisenstein series
%induced from the Siegel parabolic. \Erez{Explicate}
%In particular, they are simple and real.

There are several methods to study the poles of $L$-functions.
The results of \cite{MR2805875, MR3366033} address the analytic properties of the completed $L$-functions.
They are rather delicate and rely on Arthur's endoscopic classification (extended by Mok to unitary groups).
For our purposes we only care about the partial $L$-functions and this can be analyzed using
the Rankin--Selberg method in a rather crude form. We summarize it in the following.

\begin{lemma}
Let $\pi\in\Pi_{\cusp}(\GL(n,\A_E))$ and $\chi$ a Hecke character of $E$.
Then the following functions are entire and of finite order. (In the first three cases $E=F$.)
\begin{enumerate}
\item $s(1-s)L^S(s,\pi\times\chi,\Sym^2\times\Std_1)$.
\item $s(1-s)L^S(s,\pi\times\chi,\wedge^2\times\Std_1)$, ($n$ even).
\item $L^S(s,\pi\times\chi,\wedge^2\times\Std_1)$ ($n$ odd).
\item $s(1-s)L^S(s,\pi\times\chi,\Asai^\pm\times\Std_1)$.
\end{enumerate}
\end{lemma}

\begin{proof}
The most difficult case is the twisted symmetric square, since it involves Eisenstein series on a double cover of $\GL (n)$
(introduced in \cite{MR743816}).
This case was worked out recently by Takeda \cite{MR3161314,MR3384864} who extended earlier work by Bump--Ginzburg \cite{MR1173928}.\footnote{For our purposes we could have avoided this case by using Remark \ref{rem: weaknice is enough}.}
The other cases are easier since they involve the well-understood mirabolic Eisenstein series on $\GL(n)$.
The argument is completely standard.
For completeness we include it for the Asai $L$-function of a representation $\pi\in\Pi_{\cusp}(\GL (n, \A_E))$.
(The twisted Asai $L$-function is obtained by twisting $\pi$.)
%Let $G=\Res_{E/F}\GL (n) \times\GL (1)$ and let $r=\Asai\otimes\Std_1$ where
%\[
%\Asai:{}^L(\Res_{E/F}\GL (n))=(\GL (n,\C)\times\GL (n,\C))\rtimes_\theta\Gamma_F\rightarrow\GL(n^2,\C)
%\]
%\Erez{Recall}
%where the action $\theta$ factors through $\Gal(E/F)$ with the non-trivial element of $\Gal(E/F)$
%permuting the two $\GL (n,\C)$ factors.

%Note that for any Hecke character $\chi'$ of $E^*\bs\A_E^*$ extending $\chi$ we have
%$L^S(s,\pi\otimes\chi,\Asai\otimes\Std_1)=L^S(s,\pi',\Asai)$ where $\pi'=\pi\otimes\chi'$.
%Therefore, it is enough to consider the (untwisted) Asai $L$-function.
For any Schwartz-Bruhat function $\Phi\in\swrz(\A^n)$ let $\Eisen_\Phi$ be the normalized Eisenstein series on $\GL (n,F)\bs\GL (n,\A)$ given by
\[
\Eisen_\Phi(g,s)=\int_{\R_{>0}}\sum_{\xi\in F^n\setminus\{0\}}\Phi(\xi tg)\abs{\det tg}^s\ d^*t.
\]
Here we embed $\R\hookrightarrow F\otimes\R\hookrightarrow\A$ by $x\mapsto 1\otimes x$.
As for the Riemann zeta function, using the Poisson summation formula we have
\begin{align*}
\Eisen_\Phi(g,s)=&\int_1^\infty\sum_{\xi\in F^n\setminus\{0\}}\Phi(\xi tg)\abs{\det tg}^s\ d^*t-\frac{\Phi(0)}s
+\\&\int_1^\infty\sum_{\xi\in F^n\setminus\{0\}}\hat\Phi(\xi t\iota(g))\abs{\det tg}^{1-s}\ d^*t+
\frac{\hat{\Phi}(0)}{s-1}
\end{align*}
where $\iota(g)=\,^tg^{-1}$. In particular, the function $s(s-1)\Eisen_\Phi(g,s)$ is entire.

The Rankin--Selberg integral
\begin{equation} \label{eq: RSAsai}
\int_{\GL (n,F)\bs\GL (n,\A_F)}\varphi(g)\Eisen_\Phi(g,s)\ dg
\end{equation}
for a cuspidal automorphic form $\varphi$ on $\GL (n,E)\bs\GL (n,\A_E)$
is a twisted version of the usual integral for $\GL(n)\times\GL(n)$. It was considered in \cite{MR984899}.
Let $N_n$ be the group of upper unitriangular matrices in $\GL (n)$.
Fix a non-degenerate character $\psi_N$ of $N_n(\A_E)$ which is trivial on $N_n(\A)$.
Let
\[
W(g)=\int_{N_n(E)\bs N_n(\A_E)}\varphi(ng)\psi_N(n)\ dn
\]
be the Fourier coefficient of $\varphi$.
Assume that $\varphi$ corresponds to a factorizable vector in $\pi=\otimes\pi_v$ (as an abstract representation).
Then we can write $W(g)=\prod_vW_v(g_v)$ where $W_v$ is in the Whittaker model of $\pi_v$ (with respect to $\psi_{N_v}$),
and $W_v(e)=1$ for all $v\notin S$.
For $\Phi=\otimes\Phi_v$ and $\Re s$ large, \eqref{eq: RSAsai} unfolds to
\[
L^S(s,\pi,\Asai^+)\prod_{v\in S}I_v(s,W_v,\Phi_v)
\]
where
\[
I_v(s,W_v,\Phi_v)=\int_{N_n(F_v)\bs\GL (n,F_v)}W_v(g)\Phi_v(e_ng)\abs{\det g}^s\ dg,
\]
$e_n$ is the row vector $(0,\dots,0,1)\in F^n$.
Since $s(1-s)\Eisen_\Phi(g,s)$ is entire, it remains to show that for any $s\in\C$ we can choose data $W_v$, $\Phi_v$ such that
$I_v(s,W_v,\Phi_v)\ne0$. We can write
\begin{equation} \label{eq: IintermsofJ}
I_v(s,W_v,\Phi_v)=\int_{F_v^n}J_v(W_v,\xi,s)\Phi_v(\xi)\ d\xi,
\end{equation}
where
\begin{eqnarray*}
J_v(W_v,e_ng,s) &=& \int_{N_n(F_v)\bs P_n(F_v)}W_v(pg)\abs{\det pg}^{s-1}\ dp\\ &=&
\abs{\det g}^{s-1}\int_{K_{n-1,v}}\int_{T_{n-1}(F_v)}W_v(tkg)\abs{\det t}^{s-1}\delta_{B_{n-1}}(t)^{-1}\ dt.
\end{eqnarray*}
Here $P_n$ is the stabilizer of $e_n$, $T_{n-1}$ is the group of diagonal matrices
with $1$ in the right bottom corner, $B_{n-1}$ is the group of upper triangular matrices of
$\GL ({n-1})$ embedded in $\GL (n)$ via $b\mapsto\sm b001$ and $K_{n-1,v}$ is the standard maximal compact subgroup of $\GL (n-1,F_v)$.
Using the asymptotics of the Whittaker function
(see \cite{MR581582, MR2495561} for the $p$-adic case, \cite[Ch.~15]{MR1170566} for the archimedean case),
the function $J_v(W_v,\xi,s)$ admits a meromorphic continuation in $s$.
More precisely, in the non-archimedean case there exists
a polynomial $P_v(x)$ such that for any fixed $W_v$, $P_v(q_v^{-s})J_v(W_v,\xi,s)$ is a polynomial in $q_v^{-s}$ which is locally
constant in $\xi\in F_v^n\setminus\{0\}$. In the archimedean case, for any real number $A$ there exists a polynomial $P_v(s)$ such that
$P_v(s)J_v(W_v,\xi,s)$ is continuous in $(\xi,s)\in F_v^n\setminus\{0\}\times\{\Re s>A\}$ and holomorphic
in $s$ for $\Re s>A$. Hence, $P_v(s)J_v(W_v,\xi,s)$, $\xi\in F_v^n\setminus\{0\}$ is a normal family of analytic functions on $\Re s>A$.
Taking $\Phi_v\in\swrz(F_v^n\setminus\{0\})$ we infer that in both the archimedean and non-archimedean cases \eqref{eq: IintermsofJ} holds for all $s$.
Assume on the contrary that for some fixed $s\in\C$, $I_v(s,W_v,\Phi_v)=0$ for all $\Phi_v\in\swrz(F_v^n\setminus\{0\})$ and $W_v$.
Varying $\Phi_v$ we infer that $J_v(W_v,\xi,s)=0$ for all $\xi$. In particular,
\[
\int_{N_n(F_v)\bs P_n(F_v)}W_v(p)\abs{\det p}^{s-1}\ dp=0
\]
for all $W_v$. Since we can take $W_v$ such that its restriction to $P_v(E_v)$ is an arbitrary smooth left $(N_v(E_v),\psi_v)$-invariant
function which is compactly supported modulo $N_v(E_v)$ \cite{MR0404534, MR2733072, MR3352025} we get a contradiction.

For the twisted exterior square case one can argue in a similar way using the Jacquet--Shalika integral representation \cite{MR1044830}.
(In the even case, the pole will come from the Eisenstein series as above. In the odd case there is no Eisenstein series
and the zeta integral is entire.)
The integral $J_v$ in this case involves an extra unipotent integration, but its meromorphic continuation (with the extra uniformity
property mentioned above) follows from the argument of \cite{MR1044830}. (See \cite{1108.2200} for more details.)
Alternatively, one can use the Bump-Friedberg integral \cite{MR1159108} (which does not involve an extra unipotent integration,
and hence the argument above applies with little change) to infer that $s_2(1-s_2)L^S(s_1,\pi,\Std_n)L^S(s_2,\pi\times\chi,\wedge^2\times\Std_1)$
(and hence also $s(1-s)L^S(s,\pi\times\chi,\wedge^2\times\Std_1)$) is entire. Strictly speaking, only the case $\chi=1$ is considered,
in \cite{MR1159108} but it is a simple matter to incorporate a non-trivial $\chi$ into the integral (and the unramified calculation is
essentially the same).
%\Erez{Was it done in the literature? I couldn't find it} \Erez{Look up the odd case}

We also remark that in the case where $n$ is odd, it easily follows from the theory of Eisenstein series that $L^S(s,\pi\otimes\chi,\wedge^2\otimes\Std_1)$
is holomorphic (see \cite{MR1701344} for a more precise result).
\end{proof}

We turn to property \eqref{eq: locndbnd}. By standard properties of the Shahidi local factors we have
\[
\epsilon_v^{\Sh}(s,\pi_v\times\chi_v,\Sym^2\times\Std_1)\epsilon_v^{\Sh}(s,\pi_v\times\chi_v,\wedge^2\times\Std_1)
=\epsilon_v^{\JPSS}(s,\pi_v\times(\pi_v\otimes\chi_v)).
\]
In particular,
\[
\expcond_v^{\Sh}(\pi_v\times\chi_v,\Sym^2\times\Std_1)+\expcond_v^{\Sh}(\pi_v\times\chi_v,\wedge^2\times\Std_1)=
\expcond_v^{\JPSS}(\pi_v\times\pi_v\otimes\chi_v),
\]
and by \eqref{eq: BH} we conclude that
\[
\expcond_v^{\Sh}(\pi_v\times\chi_v,\Sym^2\times\Std_1)+\expcond_v^{\Sh}(\pi_v\times\chi_v,\wedge^2\times\Std_1)\le
n(\cond(\pi_v)+\cond(\pi_v\otimes\chi_v))\le 2n\cond(\pi_v)+n^2\cond(\chi_v).
\]
Both $\expcond_v^{\Sh}(\pi_v\times\chi_v,\Sym^2\times\Std_1)$ and $\expcond^{\Sh}(\pi_v\times\chi_v,\wedge^2\times\Std_1)$
are non-negative by \cite{MR2595008} (which relies on the validity of the local Langlands conjecture for $\GL (n)$).
Therefore, we infer \eqref{eq: locndbnd} for $\Sym^2\times\Std_1$ and $\wedge^2\times\Std_1$ from the corresponding relation
for $\Std_n$.

Similarly, if $v$ is a place of $F$ which is inert in $E$ and $w$ is the place of $E$ above $v$ then we may view
$\pi_v$ as a representation $\pi_w$ of $\GL(n,E_w)$ and we have
\[
\expcond_v^{\Sh}(\pi_v,\Asai^+)+\expcond_v^{\Sh}(\pi_v,\Asai^-)=\expcond_w^{\JPSS}(\pi_w\times\pi_w^\tau,T_{n,n})
\]
where $\tau$ is the Galois involution of $E_w/F_v$.
% and $\omega_v$ is the quadratic character of $F_v^*$ corresponding to $E_w$ by class field theory.
Thus, by \cite{MR2595008}
\[
\expcond_v^{\Sh}(\pi_v,\Asai^\pm)\le\expcond_w^{\JPSS}(\pi_w\times\pi_w^\tau,T_{n,n}),
\]
and therefore
\[
\expcond_v^{\Sh}(\pi_v,\Asai^\pm)\le n(\cond(\pi_w)+\cond(\pi_w^\tau))=2n\cond(\pi_w).
\]
In the split case, if $w_1$, $w_2$ are the places of $E$ above $v$ (with $E_{w_1},E_{w_2}\simeq F_v$)
and $\pi_v=\pi_{w_1}\otimes\pi_{w_2}$ then
\[
\epsilon_v^{\Sh}(\pi_v,\Asai^\pm)=\epsilon_v^{\Sh}(\pi_{w_1}\times\pi_{w_2},\Std_n\times\Std_n)
\le n(\cond(\pi_{w_1})+\cond(\pi_{w_2})).
\]
In both cases the relation \eqref{eq: locndbnd} for $\Asai^\pm$ follows from the case of $\Std_n$.

\subsection{}
Next we consider the case where $G=G'\times\Res_{E/F}\GL (1)$ and $r=T_{m,1}\circ(\can_{G'}\otimes\id_1)$, where $G'$ is a classical group
and $m$, $T_{m,1}$ and $\can_{G'}$ are as in the beginning of the section.
%We denote by $m$ the degree of $\can_{G'}$.
%Once again we reduce to the cuspidal case.
The $L$-functions in this case were studied by Piatetski-Shapiro and Rallis using the doubling method \cite{MR803369, MR892097, MR848553} --
see also \cite{MR2192828} for some complements.\footnote{See \cite[\S 7.2]{MR3006697} for a small correction.}
The analysis of poles was carried out in \cite{MR892097, MR1159110} and completed by Yamana in \cite{MR3211043}.
In particular, $r$ has property (FE+).
Let $\gamma_v^{\PSR}(s,\pi\times\chi,r)$ be the local factors conceived by Piatetski-Shapiro--Rallis and explicated in \cite{MR2192828}.
They satisfy properties \eqref{eq: global FE}, \GF, and \eqref{eq: comparch}.
%\Erez{(See \cite{MR3211043} for inner forms.)}

We now turn our attention to condition \eqref{eq: locndbnd}.
It is tempting to try to prove it using the definition of $\gamma_v^{\PSR}(s,\pi\times\chi,r)$,
analogously to Lemma \ref{lem: uprcondva}.
However, at this stage we are unfortunately unable to carry this idea through. Instead, we will resort to a different approach
using Arthur's work (adapted by Mok to the case of unitary groups) on functoriality for classical groups.

First, we recall the following stability result.
\begin{proposition}[\cite{MR2195117, MR2406494}] \label{prop: PSRgammastability}
Let $v$ be a non-archimedean place and $\pi_1$, $\pi_2$ irreducible representations of $G(F_v)$.
In the unitary case assume that $\pi_1$ and $\pi_2$ have the same central character.
Then there exists a constant $N=N(\pi_1,\pi_2)$ such that
\[
\gamma_v^{\PSR}(s,\pi_1\times\chi,r)=\gamma_v^{\PSR}(s,\pi_2\times\chi,r)
\]
for any character $\chi$ of $E_v^*$ such that $\cond(\chi)>N$.
\end{proposition}
We caution that unfortunately there is a mistake in \cite{MR2406494} (in both the statement and the proof),
but it is easy to fix the argument to obtain the correct statement above. We omit the details.

The argument of \cite{MR2195117} explicates $N(\pi_1,\pi_2)$. Moreover, by taking $\pi_2$ to be a representation induced
from a minimal Levi subgroup of $G$, in which case $\gamma_v^{\PSR}(s,\pi_2\times\chi,r)$ is computed explicitly in
\cite[Theorem 4]{MR2192828}, we conclude
\begin{corollary} \label{cor: PSRgammastability}
There exists a constant $C>0$ such that for any non-archimedean place $v$,
any irreducible representation $\pi$ of $G(F_v)$ and any character $\chi$ of $E_v^*$ such that $\cond(\chi)>C\level (\pi)$
we have
\[
\gamma_v^{\PSR}(s,\pi\times\chi,r)=\gamma_v^{\Tate}(s,\chi)^{m-1}\gamma_v^{\Tate}(s,\chi\chi_1).
\]
Here $\chi_1$ is trivial in the symplectic and odd orthogonal cases,
while in the even orthogonal case $\chi_1$ is the Hilbert symbol with
$(-1)^{m/2}D$, where $D$ is the discriminant of the quadratic form defining $G$
(we denote this quadratic character by $\tau_D$), and $\chi_1(x)=\omega_\pi(x/x^\tau)$ in the unitary case.
%In particular, if
%\[
%\expcond_v^{\PSR}(s,\pi\times\chi,r)\ll\cond(\chi).
%\]
\end{corollary}

Now we turn to Arthur's work on classical groups.
In order to state the result that we need, denote by $\Pi_{\aut}(\GL(n,\A))$ the set of irreducible representations of
$\GL(n,\A)$ which are weakly contained in $L^2(\GL(n,F)\bs\GL(n,\A))$.
These are the representations which are parabolically induced from some $\sigma\in\Pi_{\disc}(M(\A))$
where $M$ is a Levi subgroup of $\GL(n)$. As explained in \cite[\S 1.3]{MR3135650}, by the Jacquet--Shalika classification theorem \cite{MR618323, MR623137}
and the M{\oe}glin--Waldspurger description of the discrete spectrum of $\GL(n)$ \cite{MR1026752},
for any finite set $S$ of places of $F$, any $\pi\in\Pi_{\aut}(\GL(n,\A))$ is determined by the collection $t_{\pi_v}$, $v\notin S\cup S(\pi)$.
Given $\pi\in\Pi_{\disc}(G'(\A))$ and $\sigma\in\Pi_{\aut}(\GL (m, \A_E))$,
we say that $\sigma$ is the transfer of $\pi$ if for any $v\notin S(\pi)$, $\sigma_v$ is unramified and
$\can_{G'}(t_{\pi_v})=t_{\sigma_v}$.
By the above, this condition determines $\sigma$ and its central character is given by
\begin{equation} \label{eq: central character relation}
\omega_\sigma=\begin{cases}1,&\text{in the symplectic or odd orthogonal case},\\
\tau_D,&\text{in the even orthogonal case},\\
\omega_\pi(x/x^\tau),&\text{in the unitary case.}\end{cases}
\end{equation}

We will need the following consequence of Arthur's work (taking also \cite{MR2599005} into account).\footnote{We thank Colette M{\oe}glin for helpful discussions on Arthur's local results in \cite{MR3135650}.}
%\Erez{Define $\Pi_{\aut}(\GL (m,\A))$.}
\begin{theorem}[\cite{MR3135650}, \cite{MR3338302}]
Any $\pi\in\Pi_{\disc}(G'(\A))$ admits a transfer $\Ar(\pi)\in\Pi_{\aut}(\GL(m,\A_E))$.
Consequently, for any $n\ge1$ and any $\pi'\in\Pi_{\cusp}(\GL (n,\A_E))$ we have
\begin{equation} \label{eq: productLfunction}
L^{S(\sigma)\cup S(\pi')}(s,\pi\times\pi',r)=L^{S(\sigma)\cup S(\pi')}(s,\Ar(\pi)\times\pi',T_{m,n}).
\end{equation}
\end{theorem}

Thanks to the work of M{\oe}glin and others, a great deal is known about the local representations $\pi_v$ in terms of $\Ar(\pi)$.
For our purposes we will only need to know the preservation of $\gamma$-factors, namely that
\begin{equation} \label{eq: presofgamfactors}
\gamma^{\PSR}(s,\pi_v\times\chi_v,r)=\gamma^{\GJ}(s,\Ar(\pi)_v\times\chi_v,T_{m,1})
\end{equation}
for all $v$ and characters $\chi_v$ of $E_v^*$. If $\pi_v$ is unramified, this follows from
the multiplicativity properties of the $\gamma$-factors on both sides (\cite{MR0342495, MR2192828}).
If $\cond(\chi_v)$ is sufficiently large (with respect to $\pi_v$ and $\Ar(\pi)_v$), then
\eqref{eq: presofgamfactors} follows from the stability properties of the $\gamma$-factors
(\cite{MR787183} for the left-hand side and Corollary \ref{cor: PSRgammastability} for the right-hand side) together with the description of the
central character of $\Ar(\pi)$.
The general case of \eqref{eq: presofgamfactors} for finite $v$ follows from the special cases above by the argument of \cite[Theorem 4.1]{MR825841},
namely, by
using \eqref{eq: productLfunction} (for $n=1$)
and comparing the functional equations
\[
L^S(s,\pi\times\chi,r)=\big(\prod_{v\in S}\gamma_v^{\PSR}(s,\pi_v\times\chi_v,r)\big)L^S(1-s,\pi\times\chi,r^\vee)
\]
and
\[
L^S(s,\Ar(\pi)\times\chi,T_{m,1})=\big(\prod_{v\in S}\gamma_v^{\GJ}(s,\Ar(\pi)_v\times\chi_v,T_{m,1})\big)L^S(1-s,\Ar(\pi)\times\chi,T_{m,1}^\vee)
\]
for suitable Hecke characters $\chi$ (with $S=S(\pi)\cup S(\chi)$).
For archimedean $v$ one can argue in a similar way using Lemma \ref{lem: torus}
(with $T$ being the multiplicative group), replacing Corollary \ref{cor: PSRgammastability} by the fact that
for any $s\in\C$ we have
\[
\lim_{\abs{t}\rightarrow\infty}\gamma^{\GJ}(s,\Ar(\pi)_v\times\abs{\cdot}^{\iii t},T_{m,1})/
\gamma^{\PSR}(s,\pi_v\times\abs{\cdot}^{\iii t},T_{m,1})=1.
\]
Indeed, this follows from Stirling's formula, the description of $\gamma^{\GJ}$ and $\gamma^{\PSR}$ in the archimedean case
and \eqref{eq: central character relation}.
%Alternatively we can use the fact (also in Arthur's book \Erez{Is it?}) that the infinitesimal characters of $\Ar(\pi)_v$ and $\pi_v$ match.

Finally, we prove property \eqref{eq: locndbnd}.
By \eqref{eq: presofgamfactors} it is enough to show that
%\begin{lemma}
%Let $\pi\in\Pi_{\disc}(G'(\A))$ and let $\sigma\in\Pi_{\aut}(\GL_m(\A))$ be its transfer.
%Then $\arithcond(\sigma,\Std_m)\ll\level(\pi)$.
%\Erez{archimedean statement as well}
%\end{lemma}
%\begin{proof}
\begin{equation} \label{eq: laosi}
\cond(\Ar(\pi)_v)\ll\log_q\level(\pi_v)
\end{equation}
for every finite place $v$. We may assume that $\level(\pi_v)>0$, i.e., that $\pi_v$ is ramified,
since otherwise $\Ar(\pi)_v$ is also unramified.
We use Corollary \ref{cor: PSRgammastability} with any $\chi_v$ of conductor $C\level(\pi_v)+1$.
We obtain
\[
\expcond^{\PSR}(\pi_v\times\chi_v)\ll\log_q\level(\pi_v)
\]
and hence by \eqref{eq: presofgamfactors},
\[
\cond(\Ar(\pi)_v\otimes\chi_v)\ll\log_q\level(\pi_v).
\]
Since $\cond(\Ar(\pi)_v)\le\cond(\Ar(\pi)_v\otimes\chi_v)+m\cond(\chi_v)$ we deduce that
$\cond(\Ar(\pi)_v)\ll\log_q\level(\pi_v)$ as required.
%\end{proof}

%More precisely, we show that the conductor of $\sigma\otimes\chi$ is bounded in terms of the level of $\pi$
%for $\chi$ of conductor the level of $\pi$.

%\Erez{Remark about inner forms - was stability worked out for them? Apparently not}

%\begin{remark}
%By the argument of \cite{MR2195117}, at least in the supercuspidal case the $\gamma$ factor is closely related to the expression
%\[
%\int_Gf(g)\abs{\det(I-g)}^s\ dg
%\]
%where $f$ is matrix coefficient of $\pi$ such that $f(e)=1$.
%If $f$ is a characteristic function of $\mathbf{c}(L)\mathbf{c}(Z)$ where $\mathbf{c}$ is the Cayley transform and
%$L$ is a lattice (corresponding to the level) then the above is essentially
%\[
%\abs{\det(1+Z)}^s\int_{L}\abs{\det(Y-Z)}^s dY.
%\]

%It would be interesting to have a formula for the epsilon factor of a representation of a classical group
%(or at least for the conductor) in the spirit of \cite{MR772712, MR882297, MR1110390, MR1950297} (for the general
%linear group and its inner forms).

%Alternatively, using the Weyl integration formula and Harish-Chandra's germ expansion for the
%character of $\pi$, it might be possible to bound the exponent of the $\gamma$ factor in terms of the level by
%taking advantage of results of Waldspurger \cite{MR1361755} and DeBacker \cite{MR1914003}. \Erez{Do you want to say something more?}
%However, we will not pursue this approach here.
%\end{remark}

We also note that by \eqref{eq: archstd}, and Remark \ref{rem: archcc} for $v\in S_\infty$ we have
\begin{equation} \label{eq: laosiarch}
\param(\Ar(\pi)_v)\ll\infcond_v^{\Std_m}(\Ar(\pi)_v)^2=\infcond_v^{\PSR}(\pi_v)^2\ll\param(\pi_v)^m.
\end{equation}

\subsection{}
Finally, consider $G=G'\times\GL (n)$ where $G'$ is a quasi-split classical group and $r=T_{m,n}\circ(\can_{G'}\otimes\id_n)$.
We argue as in \S\ref{sec: RSconvda}. Namely, for $\pi\in\Pi_{\disc}(G'(\A))$ and $\pi'\in\Pi_{\disc}(\GL(n,\A_E))$ we take
\[
\gamma_v^{\JPSS\circ\Ar}(s,\pi\times\pi',r):=\gamma_v^{\JPSS}(s,\Ar(\pi)_v\times\pi'_v,T_{m,n}).
\]
By using \eqref{eq: laosi} and \eqref{eq: laosiarch}, properties \GF, \eqref{eq: archbnd} and (CC') hold for $\gamma_v^{\JPSS\circ\Ar}$ since they hold for $\gamma_v^{\JPSS}$.

\begin{remark}
The same argument will give properties (FE+) and \CC\ for $(G,r)=(G'\times G'',r=T_{m,n}\circ(\can_{G'}\otimes\id_n))$, where $G'$ is a quasi-split classical group
and $G''$ is an inner form of $\Res_{E/F}\GL (n)$.
\end{remark}

\begin{remark} \label{rem: extend to non-qs}
Once Arthur's work is extended to general classical groups and their inner forms,
Theorem \ref{thm: ABC} will hold for them as well, with the same proof.
(The stability argument of \cite{MR2195117} should carry over to inner forms without too much trouble.)
In fact, one can hope that the methods of Arthur carry over to the $\operatorname{GSpin}$ groups (as well as their inner forms).
By the previous remark, once this is done, Theorem \ref{thm: ABC} and its proof will extend to this case as well.
%\Erez{Check \cite{1409.3731}}
\end{remark}

\begin{remark}
One may contemplate whether Arthur's work (which invokes the full force of the stable twisted trace formula) is absolutely
necessary for the question at hand.
A possible different approach would be to use the Rankin--Selberg integrals that were studied in \cite{MR1357823, MR2906909}.
In the case of generic representations, a great deal is known about these $L$-functions.
In the general case, it seems that more input is necessary to address their finer analytic properties.
%\Erez{Refer to \cite{MR2275646}}
%\Erez{This should be in conjunction with Remark \ref{rem: weaknice is enough}.}
We also mention the recent preprint \cite{1601.08240}, where a more uniform approach for these $L$-functions is laid out.
It is yet to be seen whether this sheds any light on the analytic issues at hand.
\end{remark}

\section{The exceptional group $G_2$} \label{SectionExceptional}

Finally, the last of Theorem \ref{thm: ABC} follows from Theorem \ref{thm: goodLfunctinos}
(for $n_1=2$, $n_2=1$) and the following

\begin{theorem} \label{TheoremG2}
The pair
\[
(G=\GL(2),r=\Sym^3\otimes(\wedge^2)^{-1}=\Sym^3\otimes(\det)^{-1}),
\]
where $\Sym^3$ is the (four-dimensional) symmetric cube representation of $\GL (2)$,
satisfies properties (FE+) and (CC):
\end{theorem}

%Finally, consider the third symmetric square $L$-function for $\GL_2$.
Of course, it is enough to consider $\pi\in\Pi_{\cusp}(\GL(2,\A))$.
Once again by Remark \ref{rem: Shahidi} the local factors $\gamma_v^{\Sh}(s,\pi_v,r)$ are defined and satisfy
\GF\ and \eqref{eq: comparch}.
%The meromorphic continuation, functional equation and property \AF\ follows from Shahidi's results.
The poles of $L^{S(\pi)}(s,\pi,r)$ (and in fact, of the completed $L$-function) were analyzed by Kim-Shahidi in \cite{MR1726704}.
Alternatively, we can analyze the poles using the symmetric cube lift to $\GL (4)$, also due to Kim--Shahidi \cite{MR1923967}.
%(Of course, this result was not available at the time of \cite{MR1726704}.)
Even better, using Remark \ref{rem: wABC} and the fact that $r=\Std_2^\vee\otimes\Sym^2-\Std_2$,
we do not need any information about the poles of $L^{S(\pi)}(s,\pi,r)$ (but we need
to know the existence of the symmetric square lift from $\GL (2)$ to $\GL (3)$ \cite{MR533066} to infer that
$\Std_2\otimes\Sym^2$ satisfies (FE+)).

Finally, we prove \eqref{eq: locndbnd}.
Combining \cite[Proposition 2.2]{MR1002045} and \cite[Theorem 3.5]{MR1070599} (applied to the exceptional group $G_2$),
we have
\[
\gamma^{\Sh}(s,\pi,r)=\gamma^{\JPSS}(s,\pi\times\Sym^2\pi,\Std_2^\vee\times\Std_3)/\gamma^{\GJ}(s,\pi,\Std_2^\vee),
\]
where $\Sym^2\pi$ is the symmetric square lift of $\pi$ to $\GL (3)$.
Thus,
\[
\expcond^{\Sh}(\pi,r)\le\cond(\pi^\vee\times\Sym^2\pi)\le 3\cond(\pi)+2\cond(\Sym^2\pi).
\]
Also,
\[
\gamma^{\GJ}(s,\Sym^2\pi)=\gamma^{\JPSS}(s,\pi\times\pi)/\gamma^{\Tate}(s,\omega_\pi),
\]
and therefore
\[
\cond(\Sym^2)\le\cond(\pi\times\pi)\le 4\cond(\pi).
\]
We conclude that
\[
\expcond^{\Sh}(\pi,r)\le 11\cond(\pi).
\]

\begin{remark}
For higher rank exceptional groups, some of the $L$-functions on the list of \cite{MR0419366} can be studied
by Rankin--Selberg integrals, at least under a genericity assumption (e.g., \cite{MR1795291, MR1177328, MR2451220, MR1344132},
to mention a few).
One can also take into account Remark \ref{rem: wABC} to study further cases.
However, more work has to be done in order to show that other exceptional groups satisfy properties (L) and (TWN+).
We will not pursue this matter here any further.
\end{remark}

%\bibliographystyle{amsalpha}
%\bibliographystyle{alpha}
%\bibliography{../Bibfiles/all}

\providecommand{\bysame}{\leavevmode\hbox to3em{\hrulefill}\thinspace}
\providecommand{\MR}{\relax\ifhmode\unskip\space\fi MR }
\providecommand{\MRhref}[2]{%
  \href{http://www.ams.org/mathscinet-getitem?mr=#1}{#2}
}
\providecommand{\href}[2]{#2}

\end{document}